\documentclass[a4paper]{amsart}

\usepackage{amsmath,amsfonts,amssymb,amsthm}
\usepackage{hyperref}
\usepackage[capitalize]{cleveref}
\usepackage{tikz-cd}
\usepackage{enumerate}

\newtheorem{lem}{Lemma}[section]
\newtheorem{teo}[lem]{Theorem}

\newtheorem{pro}[lem]{Proposition}
\newtheorem{cor}[lem]{Corollary}
\newtheorem{claim}[lem]{Claim}

\newtheorem*{con*}{Conjecture}

\newtheorem{Conj}{Conjecture}

\theoremstyle{definition}

\newtheorem{defn}[lem]{Definition}

\theoremstyle{remark}
\newtheorem*{rem*}{Remark}

\DeclareMathOperator{\bbz}{\mathbb{Z}}
\DeclareMathOperator{\bbf}{\mathbb{F}}
\DeclareMathOperator{\Hom}{\mathrm{Hom}}

\DeclareMathOperator{\Aut}{\mathrm{Aut}}
\DeclareMathOperator{\cd}{\mathrm{cd}}
\newcommand{\Falgebra}[1]{\bbf_p[#1]}
\newcommand{\FFalgebra}[1]{\bbf_p[\![#1]\!]}
\newcommand{\G}{\mathcal G}
\newcommand{\HH}{\mathrm{H}}
\newcommand{\hpi}{\widehat{\pi}}

\title{Profinite detection of free products and free factors}
 
\author{Andrei Jaikin-Zapirain}
\address{Instituto de Ciencias Matem\'aticas, CSIC-UAM-UC3M-UCM}
\email{andrei.jaikin@icmat.es}
\author{Henrique Souza }
\address{Departamento de Matem\'aticas, Universidad Aut\'onoma de Madrid}
\email{henrique.mendesdasilva@uam.es}
\author{Pavel Zalesski}
\address{Departamento de Matem\'atica, Universidade de Bras\'ilia}
\email{pz@mat.unb.br}

\date{\today}

\begin{document}

\begin{abstract} Let \(G\) be the fundamental group of a graph of finitely generated virtually free groups with virtually cyclic edge groups. We show that \(G\) is cohomologically good if \(G\) is residually finite. If \(G\) is LERF, we prove that \(G\) splits non-trivially as a free product if and only if its profinite completion \(\widehat{G}\) splits non-trivially as a free profinite product. Moreover, we are able to detect one-ended free factors of \(G\) from \(\widehat{G}\). As an application, we deduce that any profinitely rigid word in a finitely generated free group is universally profinitely rigid.
\end{abstract}

\maketitle

\section{Introduction}

A central theme in modern infinite group theory is the extend to which the finite quotients of a group determine its algebraic structure. Problems of this sort, colloquially known as \emph{profinite rigidity} or \emph{profinite detection} problems, asks which properties of a finitely generated residually finite group \(G\) can be recovered from its profinite completion \(\widehat{G}\). In this setting, the strongest question one might ask is: given two finitely generated and residually finite groups \(G\) and \(H\) such that their profinite completions \(\widehat{G}\) and \(\widehat{H}\) are isomorphic, must \(G\) and \(H\) be themselves isomorphic?

Groups for which this holds are called profinitely rigid in the \emph{absolute sense}. Going past the realm of abelian groups, very few groups are known to be profinitely rigid in the absolute sense, with some recent known examples coming from three-dimensional hyperbolic geometry (\cite{bridson_absolute_2020,bridson_profinite_2021}), some lamplighter groups (\cite{blachar_profinite_2025}) and free meta-abelian groups (\cite{wykowski2025profiniterigidityfreemetabelian}). A major open problem in the field, due to Remeslennikov, asks whether free groups are profinitely rigid in the absolute sense  (\cite[Problem 5.48]{leifman_kourovka_1983}). A weaker notion is that of \emph{relative} profinite rigidity, where one asks whether the profinite completion \(\widehat{G}\) determines the isomorphism class of \(G\) among groups in a given class \(\mathcal{C}\) of finitely generated residually finite groups. A major achievement in this direction is due to Wilton (\cite{Wi18}), who showed that free groups are rigid in the class of hyperbolic fundamental groups of graphs of free groups with cyclic edge groups. Later, the hyperbolicity assumption was dropped by Morales in \cite{moralesProfiniteRigidityFree2024}.

Recently, two broader questions have come into consideration. We say that profinite completions \emph{detect free products} in \(\mathcal{C}\) if every finitely generated residually finite group \(G \in \mathcal{C}\) whose profinite completion \(\widehat{G}\) decomposes as a non-trivial free profinite product \(\widehat{G} = L \amalg K\) also decomposes non-trivially as a free product \(G = H*U\). We say that profinite completions \emph{detect free factors} in \(\mathcal{C}\) if for every finitely generated residually finite group \(G\) with a finitely generated subgroup \(H \leq G\) such that \(\widehat{G} \simeq \overline{H} \amalg K\) there exists a subgroup \(U \leq G\) such that \(G \simeq H * U\). 

These two properties are related, but we have no reason to believe that one implies the other. It is known that, in the class of virtually free groups, profinite completions detect free products (\cite{BPZ22}) and free factors (\cite{garridoFreeFactorsProfinite2023}). In this paper, we will consider this properties in the class of fundamental groups of graphs of virtually free groups with virtually cyclic edge groups, a class  henceforth denoted by \(\mathcal{C}\).

A group is called \emph{LERF} (\emph{locally extended residually finite}, or \emph{subgroup separable}) if every finitely generated subgroup is closed in the profinite topology. Equivalently, \(G\) is LERF if for every \(H \leq G\) finitely generated and every \(x \not\in H\) there is a finite index subgroup \(N \leq G\) containing \(H\) but not \(x\). If \(G\) is LERF, then the profinite topology of \(G\) induces the profinite topology on each of its finitely generated subgroups (\cite[Lem. 11.1.4]{ribesProfiniteGraphsGroups2017}). The main result of the paper is that, assuming the LERF condition, profinite completions detect free products in \(\mathcal{C}\).

\begin{teo}\label{thm-detects-free-products} Let \(G\) be a finitely generated LERF group that decomposes as the fundamental group of a graph of virtually free groups with virtually cyclic edge groups. If \(\widehat{G}\) decomposes as a non-trivial free profinite product, then \(G\) decomposes as a non-trivial free product.
\end{teo}

At this moment, we don't know whether the LERF condition in Theorem~\ref{thm-detects-free-products} is necessary or if it can be replaced by the condition of residual finiteness.  We also obtain the following corollary:

\begin{cor}\label{cor-relative-splitting} Let  \(G = F_1 *_{c_1 = c_2} F_2 \) be the free product of two free groups amalgamated along non-trivial proper cyclic subgroups  \( \langle c_i  \rangle  < F_i \). If  \( \widehat{G}  = L  \amalg K \) splits non-trivially as a free profinite product, then one of the  \(F_i \) splits non-trivially as a free product  \(F_i = H_i * K_i \) such that  \(c_i  \in H_i \).
\end{cor}

One of the main ingredients of the proof of Theorem~\ref{thm-detects-free-products} is the proof of a Kurosh-style decomposition for virtually free groups. 

\begin{teo}\label{thm-relative-splitting-virtually-free-groups} Let \(G\) be a finitely generated virtually free group and suppose that \(\widehat{G} \leq P = L \amalg K\). Then \[G \simeq \left(\underset{g \in L\backslash P/G}{\ast} G \cap {^gL}\right)*\left(\underset{h \in K \backslash P/ G}{\ast} G \cap {^hK}\right)*F\] where \(F\) is a free subgroup of \(G\).
\end{teo}

With respect to the problem of detecting free factors, we can prove the following statement.

\begin{teo}\label{oneendedH} Let \(G = \pi(Y,\G)\) be the fundamental group of a graph of virtually free groups with virtually cyclic edge groups. Assume that \(G\) is finitely generated and LERF. Let \(H\) be a finitely generated subgroup of \(G\) and moreover assume that either
\begin{enumerate}
    \item \(H\) has at most one end or
    \item \(H \leq {^ g\G(v)}\) is contained in some \(G\)-conjugate of a vertex group.
\end{enumerate}
If \(\overline{H}\) is a free profinite factor of \(\widehat{G}\), then \(H\) is a free factor of \(G\).
\end{teo}

Another result we obtain, which is of independent interest, concerns the property of being \emph{cohomologically good} (or \emph{cohomologically separable}); that is, if \(A\) is a finite \(G\)-module, then the restriction map on cohomology
\[
\mathrm{H}_{\mathrm{ct}}^*(\widehat{G}, A) \to \mathrm{H}^*(G, A)
\]
is an isomorphism in every degree.
 The profinite completion of a finitely generated LERF group \(G\) from \(\mathcal{C}\) has a very transparent structure: it is the profinite fundamental group of the induced graph of profinite groups given by the profinite completions of the vertex and edge groups of \(G\). This implies that they are cohomologically good (see, for instance, \cite[Thm. 3.9]{Grunewald2008}). We extend this result to finitely generated residually finite groups in the class \(\mathcal{C}\).

\begin{teo}\label{goodness} Let \(G\) be the fundamental group of a graph of virtually free groups with virtually cyclic edge groups. If \(G\) is finitely generated and residually finite, then \(G\) is cohomologically good. In particular, if \(G\) is not virtually free, then \(\widehat{G}\) has virtual cohomological dimension 2.
\end{teo}

This result was previously known in the case when \(G\) is the fundamental group of a finite graph of cyclic groups by a result of Cohen and Wykowski (\cite{CW24}).

We also present an application of Theorem~\ref{thm-relative-splitting-virtually-free-groups} concerning profinite properties of words in free groups. Following Hanani, Meiri and Puder (\cite[Def. 1.3]{hanany_orbits_2020}), we say that a word \(w\) in a finitely generated free group \(F\) is \emph{profinitely rigid} if any other word \(u \in F\) that lies in the same \(\Aut(\widehat{F})\)-orbit of \(w\) is also in the same \(\Aut(F)\)-orbit (see Section~\ref{sect:rigid-words} for an equivalent description). A profinitely rigid word \(w \in F\) is \emph{universally profinitely rigid} if it is profinitely rigid in any free extension \(F * H\) of \(F\), where \(H\) is another finitely generated free group.

It has been conjectured that primitive words, those that generate a free factor of \(F\), are profinitely rigid, and this has been confirmed by Puder and Parzanchesvki in \cite{Puder2014} who showed that they are also universally profinitely rigid. Another proof was given by Wilton in \cite{Wi18}. Later, it was shown in \cite{hanany_orbits_2020} that if \(w \in F\) is profinitely rigid then so is \(w^d\) for every \(d \in \bbz\), as well as proving that a basic commutator \([x,y]\) of two elements of a basis of \(F\) is profinitely rigid. In \cite{wilton_profinite_2021}, it was shown that the surface words \(x_1^2x_2^2\cdots x_n^2\) and \([x_1,x_2]\cdots[x_{2g-1},x_{2g}]\) are profinitely rigid in \(F_n\) and \(F_{2g}\) respectively. We prove that the notions of profinite rigidity and universal profinite rigidity are equivalent:

\begin{teo}\label{theorem-profinitely-rigid-words} Let \(F\) be a finitely generated free group and \(w \in F\). Then, \(w\) is profinitely rigid in \(F\) if and only if it is profinitely rigid in \(F * H\) for every finitely generated free group \(H\).
\end{teo}

Observe that any non-trivial element is clearly profinitely rigid in the free group of rank one, and hence Theorem~\ref{theorem-profinitely-rigid-words} implies, in particular, the profinite rigidity of powers of primitive words.

\medskip

Finally, we describe the structure of the paper. We collect in Section~\ref{sect:prelim} essential results about the class \(\mathcal{C}\), abstract and profinite Bass--Serre theories, small cancellation groups and completed group algebras that will be needed throughout the paper. In Section~\ref{sect:profcompl} we describe the profinite completion \(\widehat{G}\) of groups in \(\mathcal{C}\) splitting as graphs of groups with a single vertex, and prove \cref{goodness}. We also pose a conjecture about the profinite completion of certain one-relator groups that is related to Remeslennikov's question (see \cref{Baby}).

The core of Section~\ref{sect:relative} is the proof of \cref{thm-relative-splitting-virtually-free-groups}, and the core of Section~\ref{sect:main-thm} is the proof of a slightly more precise version of \cref{thm-detects-free-products}, from which both \cref{oneendedH} and \cref{cor-relative-splitting} also follows easily. The application to the profinite rigidity of words, \cref{theorem-profinitely-rigid-words}, is explained and proven in Section~\ref{sect:rigid-words}. We finish the paper discussing possible generalizations of these results to the class \(\mathcal{L}\) of limit groups in Section~\ref{sect:limit}.

\subsection*{Acknowledgements}

We would like to thank Henry Wilton for useful discussions on the subject. We are also grateful to Macarena Arenas for providing references on small cancellation groups, and to Ashot Minasyan for references on one-ended one-relator groups.

The first and third authors would like to thank the Isaac Newton Institute for Mathematical Sciences, Cambridge, for support and hospitality during the programme ``Operators, Graphs, Groups" where work on this paper was undertaken. This work was supported by EPSRC grant no EP/Z000580/1 and the Simons Foundation, Award SFI-MPS-T-Institutes-00006117.
The work of the first and second authors is partially supported by the grant  {PID2020-114032GB-I00}  of the Ministry of Science, Innovation and Universities of Spain of the third author by FAPDF and CNPq.

\section{Preliminaries}\label{sect:prelim}

\subsection{Notation}

The disjoint union of two sets \(A\) and \(B\) is denoted by \(A \sqcup B\). For two group elements \(g\,, h \in G\), conjugation is a left action: \(^ gh = ghg^{-1}\). If \(G\) is a profinite group and \(H\) is a closed subgroup of \(G\), then the index \([G\colon H]\) is a \emph{supernatural number}, that is, a formal product \(\prod p^{a_p}\) over all the primes \(p\) such that \(a_p \in \bbz_{\geq 0}\cup \{\infty\}\) (\cite[Sec. 2.3]{ribesProfiniteGraphsGroups2017}). A closed subgroup is open if and only if \([G\colon H]\) is a natural number, that is, \(a_p \in \bbz\) for all \(p\) and almost all \(a_p\) are zero. The order of a profinite group \(G\) is the supernatural number \(|G| =[G\colon \{1\}]\).

We denote the (abstract) free product of two groups \(H\) and \(U\) by \(H * U\), and the free profinite product of two profinite groups \(L\) and \(K\) by \(H \amalg K\). We refer the reader to \cite[Ch. 9]{RibesZalesskii2010} for the definition and basic properties of free profinite products. If \(G\) is an abstract group, we denote by \(\widehat{G} = \varprojlim_{N \unlhd_{\mathrm{f.i.}} G} G/N\) its profinite completion, the inverse limit being taken over all normal subgroups \(N \unlhd G\) such that the index \([G\colon N]\) is finite.

If \(G\) is an abstact (resp. profinite) group, we denote by \(\HH^i(G,A)\) (resp. \(\HH^i_{\mathrm{ct}}(G,A)\)) the \(i\)-th cohomology group (resp. \(i\)-th continuous cohomology group) of a \(G\)-module (resp. discrete \(G\)-module) \(A\). We say that \(\cd_p(G) \leq n\) if \(\HH^i(G,A)\) (resp. \(\HH^i_{\mathrm{ct}}(G,A)\)) vanishes for every \(i > n\) and every \(p\)-primary \(G\)-module (resp. \(p\)-primary discrete \(G\)-module) \(M\).

We denote by \(\bbf_p\) the field of \(p\) elements. If \(X\) is any set and \(R\) is a commutative unital ring, \(R[X]\) denotes the free \(R\)-module with basis \(X\). For any group \(G\) (resp. \(G\)-set \(X\)), \(R[G]\) (resp. \(R[X]\)) has naturally the structure of an \(R\)-algebra (resp. \(R[X]\)-module) called the group algebra of \(G\) over \(R\). If \(G\) is a profinite group, the completed group algebra of \(G\) over \(\bbf_p\) is the profinite ring \(\FFalgebra{G} = \varprojlim_{U \unlhd_o G} \Falgebra{G/U}\) where \(U\) runs through the open subgroups of \(G\). If \(G\) is a profinite group acting continuously on a profinite set \(X = \varprojlim X_i\), where each \(X_i\) is a finite \(G\)-set, we denote by \(\FFalgebra{X}\) the \(\FFalgebra{G}\)-module \(\varprojlim \Falgebra{X_i}\). For any \(G\)-module \(M\), we denote by \(M^G\) the submodule of \(G\)-fixed points.

If \(\epsilon\colon A \to R\) is the augmentation map of an augmented \(R\)-algebra \(A\), we denote by \(I_A\) its kernel \(\ker \epsilon\), called the augmentation ideal of \(A\). It is well known that the augmentation ideal of \(\Falgebra{G}\) (resp. \(\FFalgebra{G}\)) is generated (resp. topologically generated) as an ideal by the elements \(x_i - 1\) where \(\{x_i\}\) is any generating (resp. topologically generating) set of \(G\). If \(H\) is a subgroup of \(G\) and \(M\) is a left \(\Falgebra{H}\)-module, we denote by \({^GM}\) the induced left \(\Falgebra{G}\)-module \(\Falgebra{G} \otimes_{\Falgebra{H}} M\). When \(G\), \(H\) and \(M\) are profinite, \(\widehat{^GM}\) denotes the profinite induced module \(\FFalgebra{G} \widehat{\otimes}_{\FFalgebra{H}} M\) where \(\widehat{\otimes}\) denotes the completed tensor product (\cite[Sec. 5.5]{RibesZalesskii2010}).

If \(F\) is a free group with basis \(x_1\,,\ldots\,,x_d\) and \(a \in F\) is any word, the \emph{Fox derivative} \(\frac{\partial a}{\partial x_i}\) denotes the unique element of \(\bbz[F]\) satisfying \[a - 1 = \frac{\partial a}{\partial x_1}(x_1-1) + \cdots + \frac{\partial a}{\partial x_d}(x_d-1)\,.\]

\subsection{The fundamental group of a graph of groups}

A \emph{graph} \(Y = V(Y) \sqcup E(Y)\) is the disjoint union of a set \(V = V(Y)\) of \emph{vertices} and a set \(E = E(Y)\) of \emph{edges} together with two maps \(i,t\colon E \to V\) that associate to each edge \(e \in E\) its \emph{initial} \(i(e)\) and terminal \(t(e)\) vertices. Loops \(i(e) = t(e)\) are allowed. In other words, we consider directed multigraphs. We say that \(Y\) is connected if it is connected as an undirected multigraph. A \emph{tree} is a connected graph which has no undirected cycles. By Zorn's lemma, every connected graph has a maximal subtree.

A \emph{graph of groups} \((Y, \G)\) consists of a graph \(Y\), vertex groups \(\{\G(v) \colon v \in V(Y)\}\) and edge groups \(\{\G(e)\colon e \in E(Y)\}\) together with two homomorphisms \(i_e\,, t_e\colon \G(e) \to \G(i(e))\,, \G(t(e))\) for each edge \(e \in E(Y)\). If all the homomorphisms \(i_e\,, t_e\) are injective, the graph of groups is called \emph{faithful}. We will only consider faithful graphs of groups.

Given a connected graph of groups \((Y,\G)\) and a maximal subtree \(T\) of \(Y\), the \emph{fundamental group} \(G = \pi(Y,\G,T)\) of \((Y,\G)\) with respect to \(T\) is the group with presentation:
\begin{itemize}
    \item generators \(\{\G(v), q_e\colon v \in V(Y)\,, e \in E(Y)\}\);
    \item relations of each \(\G(v)\) for each vertex \(v \in V(Y)\);
    \item relations \({^{q_e}i_e(g)} =q_ei_e(g)q_e^{-1} = t_e(g)\) for each \(e \in E(Y)\) and \(g \in \G(e)\);
    \item relations \(q_e = 1\) for each \(e \in T\).
\end{itemize}

One of the first results of Bass--Serre theory (see \cite[Sec. I.7]{dicksGroupsActingGraphs1989}) is that any other maximal subtree yields an isomorphic fundamental group. Hence, \(\pi(Y,\G)\) is well defined up to isomorphism. Another result (\cite[Cor. I.7.5]{dicksGroupsActingGraphs1989}) states that if \((Y,\G)\) is connected and faithful, then the canonical maps \(\G(v) \to \pi(Y,\G)\) are injective and we may identify each \(\G(v)\) with its image.

Let \(G\) act faithfully on a tree \(X\) and let \(S \subset X\) be a \emph{connected transversal} for the \(G\)-action. This is a subset of \(X\) (in general, not a subgraph) that is in bijection with \(G / X\), any two of whose vertices are connected by a path in \(S\) and such that \(i(e) \in S\) for each edge \(e \in S\). For each edge \(e \in S\), there is a unique vertex of \(S\) in the same \(G\)-orbit as \(t(e)\), so we choose \(q_e \in G\) such that \(q_e\cdot t(e) \in S\), taking \(q_e = 1\) if \(t(e) \in S\). The collection \((q_e\colon e \in E(Y) \cap S)\) is a \emph{connecting family} for \(S\). For each \(s \in S\), denote by \(G_s\) its stabilizer in \(G\). Note that \(G_{q_e\cdot t(e)} = {^{q_e} G_{t_e}}\) and that for every edge \(e \in S\), \(G_e = G_{i(e)} \cap G_{t(e)} = G_{i(e)} \cap {^{q_e^{-1}}G_{q_e\cdot t(e)}}\), so that \({^{q_e}G_e} \leq G_{q_e\cdot t(e)}\). These are exactly the conjugation relations that appear in the definition of the fundamental group of a graph of groups.

The structure theorem of Bass--Serre theory says that \(G \simeq \pi(G/ X, \G)\), where \(\G(s) = G_s\) for \(s \in S\) and \(i_e\,, t_e\colon \G(e) \to \G(i(e))\,, \G(q_e\cdot t(e))\) are given by \(g \mapsto g, {^{q_e}g}\) respectively. Moreover, the isomorphism \(\varphi\colon \pi(G/ X, \G) \to G\) is defined by the inclusion maps \(\G(v) \to G\) and \(q_e \mapsto q_e\) for each vertex \(v\) and each edge \(e\).

Let \(X\) be a \(G\)-tree with finite edge stabilizers. We say that an edge \(e\) of \(X\) is \emph{compressible} if two vertices of \(e\) are not in the same \(G\)-orbit and, for at least one vertex \(v\) of \(e\), we have \(G_v = G_e\). Otherwise, we say that the edge \(e\) is \emph{incompressible}. The \(G\)-tree \(X\) is \emph{incompressible} if all its edges are incompressible (see \cite[Lem. III.7.2]{dicksGroupsActingGraphs1989} for an alternative definition).

\subsection{Derivations}

Let \(G\) be a group and \(A\) a \(G\)-module. A \emph{derivation} is a map \(d\colon G \to A\) such that \[d(gh) = d(g) + g\cdot d(h)\quad\text{for all }g\,, h \in G\,.\] A derivation \(d\) is called \emph{inner} if there exists an element \(a \in A\) such that \[d(g) = (g-1)\cdot a\quad\text{for all }g \in G\,.\] The set of derivations from \(G\) to \(A\) is denoted by \(\mathrm{C}^1(G,A)\), and the set of inner derivations by \(\mathrm{B}^1(G,A)\). We can identify the quotient \(\mathrm{C}^1(G,A)/\mathrm{B}^1(G,A)\) with the first cohomology group \(\HH^1(G,A)\).

A direct application of the definition gives the following lemma.

\begin{lem}\label{claim-conjugate-derivations} Let \(d\colon G \to A\) be a derivation from a group \(G\) to a \(G\)-module \(A\), and let \(h \in G\). If \(d(h) = (h-1)a\) for some \(a \in A\), then \[d({^g h}) = ({^gh}-1)\cdot (ga-d(g))\,.\]
\end{lem}

As a consequence, if \(d\) is an inner derivation when restricted to a subgroup \(H \leq G\), then \(d\) is an inner derivation when restricted to any \(G\)-conjugate of \(H\). This applies in particular to \(H = \ker d = \{g \in G\colon d(g) = 0\}\). We will also make use of the following theorem of Dicks and Dunwoody (\cite[Cor. IV.2.8 and Rem. IV.2.9]{dicksGroupsActingGraphs1989}):

\begin{teo}\label{DD} If \(d\colon G \to P\) is a derivation to a projective \(\Falgebra{G}\)-module \(P\), then there exists a \(G\)-tree \(X\) with finite edge stabilizers and in which a subgroup \(H \leq G\) stabilizes a vertex if and only if the restriction of \(d\) to \(H\) is inner. In particular, if \(G\) is torsion free, this characterizes \(H\) as a free factor.
\end{teo}

\subsection{The profinite fundamental group of a graph  of profinite groups}

Given a finite graph of profinite groups \((Y,\G)\), the \emph{profinite fundamental group} \(\hpi(Y,\G)\) is defined as the profinite group with the same generators and relations as its abstract analog \(\pi(Y,\G)\) (see \cite[Sec. 6.2]{ribesProfiniteGraphsGroups2017}). There is one important difference with the abstact case: in the profinite setting, the canonical homomorphisms \(\G(v) \to \hpi(Y,\G)\) are not embeddings in general, but one can always replace them by their images in \(\hpi(Y,\G)\) and then the embedding will hold. From now on, we assume that the vertex subgroups \(\G(v)\) embed into \(\hpi(Y,\G)\). If all the vertex subgroups are finitely generated, then \(\hpi(Y,\G)\) is isomorphic to the profinite completion of \(\pi(Y,\G)\).

Given the profinite fundamental group \(G = \pi_1(Y,\G)\) of a finite graph of profinite groups, by \cite[Thm. 6.5.2]{ribesProfiniteGraphsGroups2017} we have the following exact sequence of left \(\FFalgebra{G}\)-modules:
\begin{equation}\label{exactseq} 0 \to \bigoplus_{e \in E(Y)} \FFalgebra{G/ \G(e)} \to \bigoplus_{v \in V(Y)} \FFalgebra{G/ \G(v)} \to \bbf_p \to 0\,.
\end{equation}

For a continuous \(G\)-module \(M\), there exists a Mayer-Vietoris sequence associated to \((Y,\G)\) (see \cite[Thm. 8.4.7]{Wilkes}):
\begin{equation}\label{Mayer-Vietoris}
\cdots \to \HH^n_{\mathrm{ct}}(G,M) \to \bigoplus_{v \in V(Y)} \HH^n_{\mathrm{ct}}(\G(v),M) \to \bigoplus_{e \in E(Y)} \HH^n_{\mathrm{ct}}(\G(e),M) \to \HH^{n+1}_{\mathrm{ct}}(G,M) \to \cdots
\end{equation}

Let \((Y,\G)\) be a graph of abstract groups and assume that \(G = \pi(Y,\G)\) is residually finite. We denote by \((Y,\overline{\G})\) the graph of groups with \(\overline{\G}(x) = \overline{\G(x)}\) the closure of \(\G(x)\) in \(\widehat{G}\). If all the vertex and edge subgroups \(\G(x)\) are finitely generated, then we have \[\hpi(Y,\overline{\G}) \simeq \widehat{\pi(Y,\G)}\,.\]
We say that a graph of abstract groups \((Y,\G)\) is \emph{efficient} if each vertex and edge group \(\G(x)\) is closed in the profinite topology of \(\pi(Y,\G)\) and moreover this topology induces the full profinite topology on each \(\G(x)\). For example, this occurs if \(\pi(Y,\G)\) is LERF, \(Y\) is finite and all the vertex and edge subgroups \(\G(x)\) are finitely generated.

The simplest example of the profinite fundamental group of a finite graph of profinite groups is a free profinite product: if all edge groups are trivial, then the fundamental group is the free profinite product of the vertex groups and a free profinite subgroup. We will need the following lemmas.

\begin{lem}[{\cite[Prop. 9.1.12]{RibesZalesskii2010}}]\label{normalizers} Let \(G_1\,,\ldots\,, G_n\) be profinite groups and \(G = G_1 \amalg \cdots \amalg G_n\) be their free profinite product. Then \(G_i \cap {^g G_i} =  \{1\}\) for every \(x \in G \smallsetminus G_i\).
\end{lem}

\begin{lem}\label{elliptic} Let \(P = K \amalg L\) be a free profinite product.
\begin{enumerate}[(i)]
    \item \(I_{\FFalgebra{P}} \simeq \widehat{^{P}I_{\FFalgebra{K}}} \oplus \widehat{^{P}I_{\FFalgebra{L}}}\)
    \item Let \(U \leq P\) be a virtually procyclic group. If \(U \cap K\) is infinite, then \(U \leq K\).
\end{enumerate}
\end{lem}

\begin{proof} (i) In the case of a free profinite product, the exact sequence \eqref{exactseq} has the following form: \[0 \to \FFalgebra{P} \to \FFalgebra{P/ K} \oplus \FFalgebra{P / L} \to \bbf_p \to 0\,.\] Hence, \(I_{\FFalgebra{P}} \simeq \widehat{^{P}I_{\FFalgebra{K}}} \oplus \widehat{^{P}I_{\FFalgebra{L}}}\).

(ii) Since \(U \cap K\) is infinite, it contains a non-trivial infinite normal procyclic subgroup \(M \unlhd U\). However, the normalizer \(N_P(M)\) is a subgroup of \(K\) by \cite[Cor. 7.1.5(a)]{ribesProfiniteGraphsGroups2017}. This finishes the proof.
\end{proof}

\begin{lem}\label{projectiveintersection} Let \(p\) be a prime and \(P = \hpi(Y,\G)\) be the profinite fundamental group of a finite graph of profinite groups. Let \(U\) be a torsion-free closed subgroup of \(P\) and assume that
\begin{enumerate}
    \item \(\cd_p({^g U} \cap \G(e)) = 0\) for all edges \(e\) and every \(g \in P\);
    \item \(\cd_p({^g U} \cap \G(v)) \leq 1\) for all vertices \(v\) and every \(g \in P\).
\end{enumerate}
Then \(\cd_p(U) \leq 1\).
\end{lem}
\begin{proof} It is the subject of \cite[Cor. 9.3.2]{ribesProfiniteGraphsGroups2017}.
\end{proof}

\subsection{The class \texorpdfstring{\(\mathcal{C}\)}{C}}

In this paper, \(\mathcal{C}\) denotes the class of fundamental groups of graphs of virtually free groups with virtually cyclic edge groups. Observe that \(\mathcal{C}\) is closed under taking subgroups and free products. Moreover, torsion-free subgroups in \(\mathcal{C}\) are the fundamental groups of graphs of free groups with cyclic edge groups.

If \(G\) belongs to \(\mathcal{C}\) and is finitely generated, then it is also the fundamental group of a finite graph of finitely generated virtually free groups with virtually cyclic edge groups. Indeed, let \(X\) be a faithful \(G\)-tree with virtually free vertex stabilizers and virtually cyclic edge stabilizers. The convex hull \(X'\) of the \(G\)-orbit of any vertex is a \(G\)-invariant subtree of \(X\) such that \(Y = G/ X'\) is finite. As the edge stabilizers of the \(G\)-action are virtually cyclic, their \(\ell^2\)-Betti numbers vanish. If \(S\) is a connected transversal for the \(G\)-action on \(X'\), by \cite[Thm. 1.5]{Chatterji2021}, we have \begin{align*}b_1^{(2)}(G) + \sum_{v \in V(X')\cap S} \frac{1}{G_v} &= \sum_{v \in V(X')\cap S} b_1^{(2)}(G_v) + \sum_{e \in E(X')\cap S} \frac{1}{|G_e|} + \frac{1}{|G|}\\&\geq \sum_{v \in V(X')\cap S} b_1^{(2)}(G_v)\,. \end{align*}
Thus, all the groups \(G_v\) are finitely generated and the graph of groups associated to the \(G\)-action on \(X'\) is the desired decomposition.

\subsection{Residually finiteness and LERFness in the class \texorpdfstring{\(\mathcal{C}\)}{C}}

We begin this section by showing that the residually finite groups in the class \(\mathcal{C}\) contain a torsion-free subgroup, which will be the fundamental group of a graph of groups with free vertex groups and cyclic edge groups.

\begin{lem}\label{lem-residually-finite} Let \(G = \pi(Y, \G)\) be the fundamental group of a finite graph of virtually finitely generated free groups with virtually cyclic edge groups. If \(G\) is residually finite, then it is virtually torsion-free.
\end{lem}
\begin{proof} Each vertex group \(\G(v)\) can be written as the fundamental group of a finite graph of finite groups. Let \(x_{v,i}\) denote the elements in all the finite vertex groups arising in this decomposition. Since \(G = \pi(Y, \G)\) is residually finite, there exists a finite-index normal subgroup \(N \unlhd G \) such that all the elements \(x_{v,i}\) remain distinct modulo \(N\).

Consider the action of \(N\) on the Bass–-Serre tree associated to \((Y, \G)\). This action defines a graph of groups decomposition for \(N\), with virtually cyclic edge groups and vertex groups of the form \(N \cap {^g\mathcal{G}(v)}\) for some \(g \in G\). Each such vertex group is a normal subgroup of the conjugate \({^g\mathcal{G}(v)}\) and intersects trivially the finite vertex groups within that conjugate, since the corresponding elements were avoided by the choice of \(N\).

It follows that \(N \cap {^g\mathcal{G}(v)}\) is torsion free and  hence is a free group. Therefore, \(N\) is the fundamental group of a finite graph of finitely generated free groups with cyclic edge groups, which is torsion-free.
\end{proof}

A characterization of when the fundamental group of a graph of finitely generated free groups with cyclic edge groups is residually finite has been recently established in \cite{AM24}. Recall that the Baumslag--Solitar group \[BS(n,m) = \langle a,b \mid ab^na^{-1} = b^m \rangle\] is residually finite if and only if either \(|m| = 1\), \(|n| = 1\) or \(|m| = |n|\) (\cite[Thm. 1]{moldavanskii_residual_2018}). Moreover, \(BS(n,m)\) is LERF if and only if \(|m| = |n|\) (\cite[Sec. 4]{moldavanskii_residual_2018}).

\begin{teo}[{\cite[Thm. A, Prop. 3.6]{AM24}}]\label{res-free} Let \(G\) be a finitely generated torsion-free group from the class \(\mathcal{C}\). Then, the following are equivalent:
\begin{enumerate}
    \item \(G\) is residually finite.
    \item All Baumslag--Solitar subgroups of \(G\) are residually finite.
\end{enumerate}
\end{teo}

We say that an element \(g\) of infinite order in a group \(G\) is \emph{unbalanced} with respect to \(h \in G\) if there exist \(n, m \in \mathbb{Z} \smallsetminus \{0\}\) with \(|n| \neq |m|\) such that
\[hg^n h^{-1} = g^m.\] Combining Wise’s characterization from \cite[Thm. 5.1]{wiseSubgroupSeparabilityGraphs2000} of when the fundamental group of a graph of finitely generated free groups with cyclic edge groups is LERF with \cite[Prop. 3.6]{AM24}, we arrive at the following theorem.
 
\begin{teo}\label{Wise}
Let \(G\) be a finitely generated torsion-free group from the class \(\mathcal{C}\). The following are equivalent:
\begin{enumerate}
    \item \(G\) is LERF.
    \item All Baumslag--Solitar subgroups of \(G\) are LERF.
    \item \(G\) does not contain any unbalanced elements.
\end{enumerate}
\end{teo}

Combining the previous two theorems we obtain the following corollary.

\begin{cor}\label{RF-LERF} Let \(G\) be a finitely generated residually finite group of \(\mathcal{C}\). The following are equivalent:
\begin{enumerate}
    \item \(G\) is LERF.
    \item All Baumslag--Solitar subgroups of \(G\) are LERF.
    \item \(G\) does not contain any unbalanced elements.
\end{enumerate}
\end{cor}
\begin{proof} Observe that \(G\) contains a finite index normal torsion-free subgroup \(H\) from the class \(\mathcal{C}\), and \(G\) is LERF if and only if so is \(H\). Moreover, \(G\) is residually finite, so the only Baumslag-Solitar subgroups it may contain are \(BS(n,n)\) and \(BS(1,n)\), and the only non-LERF ones are the latter for \(|n| > 1\).

If condition (2) holds, then the same holds for \(H\), so \(H\) (and thus \(G\)) is LERF by Theorem~\ref{Wise}. Conversely, if \(G\) is LERF, then so must be all of its Baumslag--Solitar subgroups.

If \(G\) does not contain any unbalaced elements, then neither does \(H\), both of which are then LERF by Theorem~\ref{Wise}. If \(G\) contains an unbalaced element, say \(hg^nh^{-1} = g^m\) with \(|n| \neq |m|\), by writing \(k = [G\colon H]\) we may take \(y = h^k \in H\) to see that \(x = g^k\) is an unbalaced element of \(H\):
\begin{align*}
    yx^{n^k}y^{-1} &= h^kg^{kn^k}h^{-k}\\
    &= \left(\underbrace{h(\cdots (h(h(hg^nh^{-1})h^{-1})^nh^{-1})^n\cdots )^nh^{-1}}_{k \text{ times}}\right)^k\\
    &= g^{m^kk} = x^{m^k}\,.
\end{align*}
Hence, neither \(H\) nor \(G\) can be LERF.
\end{proof}

We shall need later the following lemma.

\begin{lem}\label{BSsubgroups} Let \(G=(Y,\G)\) be the fundamental group of a finite graph of virtually free groups with virtually cyclic edge groups. If \(B\) is a Baumslag--Solitar subgroup of \(G\), then \(B\) contains an element conjugate to \(\G(e)\) for some edge \(e\) of \(Y\). Moreover, if \(G\) is LERF and \(\widehat{G}\) splits as a free profinite product \(K\amalg L\), then \(B\) is contained in a conjugate of \(L\) or \(K\).
\end{lem}

\begin{proof} First note that $B$ is torsion free. So if $B$ does not intersect any edge subgroup, it is a free product of virtually free groups, a contradiction. 
Thus  $B$ intersect some edge group non-trivially. To prove the second part of the statement observe that the closure $\overline B \simeq \widehat{B}$ of $B$ in $\widehat G$ contains a central cyclic subgroup as \(B \simeq BS(n,\pm n)\) must be LERF and surjects onto \(\widehat{\bbz}\), and so by \cite[Prop. 4.2.9]{ribesProfiniteGraphsGroups2017} must be contained in a conjugate of $K$ or $L$. 
\end{proof}

\subsection{On a result of Wilton}

Wilton~\cite{Wi18} showed that non-free hyperbolic groups in the class \(\mathcal{C}\) virtually retract onto a surface group. Combining this with a result of Bestvina and Feighn~\cite{BF92}, which characterizes hyperbolic groups in \(\mathcal{C}\) as those not containing Baumslag–-Solitar subgroups, we obtain the following theorem:

\begin{teo}\label{morales}
Let \(G\) be a finitely generated torsion-free group in the class \(\mathcal{C}\) that does not contain any Baumslag-–Solitar subgroup. Then \(G\) virtually retracts onto a surface group.
\end{teo}

The theorem implies the following consequence for profinite completions of groups in \( \mathcal{C} \).

\begin{cor}\label{wilton}
Let \(G\) be a finitely generated group from the class \(\mathcal{C}\) that does not contain any Baumslag–-Solitar subgroup. If for some prime \(p\) the profinite completion \(\widehat{G}\) contains an open subgroup \(U\) with \(\cd_p(U) = 1\), then \(G\) is virtually free.
\end{cor}
\begin{proof} By \cref{res-free}, \(G\) is residually finite. Hence, by Lemma~\ref{lem-residually-finite}, \(G\) contains a torsion free subgroup \(H\) of finite index. If \(H\) is not free, by \cref{morales}, \(H\) must contain a surface subgroup \(S\) as a virtual retract. Therefore, the profinite topology of \(G\) induces the full profinite topology on \(S\) (this also follows from \(G\) being LERF), and hence an embedding \(\widehat{S} \to \widehat G\). This shows that \(\mathrm{vcd_p}(\widehat G) \geq 2\).
\end{proof}

\subsection{Fixed points of profinite actions}

We will need the following known result, the proof of which can be found in \cite[Lem 2.2]{garridoFreeFactorsProfinite2023}.

\begin{lem}\label{fixedpoints}
Let \(G\) be a profinite group with closed subgroups \(A, B \leq G\), and let \(p\) be a prime number. Then \[\FFalgebra{G/B}^A = \{0\} \quad \text{if and only if} \quad p^\infty \mid [A : A \cap {^gB}] \text{ for every } g \in G.\]
\end{lem}

As a particular case of the above lemma, we will also need the following description of how the augmentation map splits in the case of a semisimple completed group algebra of a procyclic group:

\begin{lem}\label{lem-splitting-augmentation} Let \(C = \overline{\langle a \rangle}\) be a procyclic group and \(p\) a prime that does not divide the order of \(C\). Then, the space \(\FFalgebra{C}^C\) of fixed points is one-dimensional and the augmentation map \(\epsilon\colon \FFalgebra{C}^C \to \bbf_p\) is an isomorphism of \(\FFalgebra{C}\)-modules.
\end{lem}
\begin{proof} We have \(C = \varprojlim_{n \in I} C/nC\) where \(I\) is the set of natural divisors of \(|C|\). Hence, \(\FFalgebra{C} = \varprojlim \Falgebra{C/nC}\) and \(\FFalgebra{C}^C = \varprojlim \Falgebra{C/nC}^C\). For finite groups of order coprime to \(p\) it is well known that \[\Falgebra{C/nC}^C = \bbf_p\left(\frac{1 + a + a^2 + \cdots +a^{n-1}}{n}\right)\] is one dimensional, spanned by the trace element of the group algebra which splits the augmentation map \(\epsilon\colon \Falgebra{C/nC} \to \bbf_p\). Therefore, it suffices to prove that \(\FFalgebra{C}^C\) is non-trivial, for which it suffices to show that the maps \(\Falgebra{C/nC}^C \to \Falgebra{C/mC}^C\) are surjective when \(n = mk\) for some \(k > 1\). However, the image of the trace element of \(\Falgebra{C/nC}\) in \(\Falgebra{C/mC}\) is \((1+\cdots+a^{m-1})/m\), so the map is surjective as desired.
\end{proof}

In accordance with the case of finite groups, in the setting of Lemma~\ref{lem-splitting-augmentation} we will denote by \(\mathrm{tr}_C\) the unique element in \(\FFalgebra{C}^C\) that maps to \(1 \in \bbf_p\) under the augmentation map, henceforth called the trace element of \(\FFalgebra{C}\). As in the proof of the Lemma, observe that the trace if functorial, that is, the image of \(\mathrm{tr}_C\) in \(\Falgebra{C/nC}\) is precisely \(\mathrm{tr}_{C/nC}\).

\subsection{Flatness of the completed algebra of the profinite completion}

One of the crucial tools in the proof of our main result is the following result, which slightly generalizes \cite[Thm. 5.6]{garridoFreeFactorsProfinite2023}.

\begin{pro}\label{flatness} 
Let \(G\) be a finitely generated virtually free group. If \(P\) is a profinite group containing \(\widehat{G}\), the left \(\Falgebra{G}\)-module \(\FFalgebra{P}\) is isomorphic to a direct union of free left \(\Falgebra{G}\)-modules.
\end{pro}
\begin{proof} Let \(F \leq G\) be a normal free subgroup of finite index. By \cite[Cor. 7.11.8]{cohnFreeIdealRings2006}, the algebra \(\Falgebra{F}\) is a free ideal ring, that is, any left or right ideal is a free as an \(\Falgebra{F}\)-module. Moreover, it was shown in \cite[Prop. 5.3]{garridoFreeFactorsProfinite2023} that \(\FFalgebra{\widehat{F}}\) is flat as an \(\Falgebra{F}\)-module. Observe that \(\widehat{F}\) is isomorphic to the closure \(\overline{F}\) of \(F\) as a subgroup of \(G\) since the index \([G\colon F]\) is finite. We will first prove that \(\FFalgebra{P}\) is flat as an abstract \(\FFalgebra{\widehat{F}}\)-module.

Observe that $\FFalgebra{\widehat{F}}$ is semi-hereditary: if $I \leq \FFalgebra{\widehat{F}}$ is a finitely generated left or right ideal, then it is closed and profinite. Since \(\FFalgebra{\widehat{F}}\) has global dimension \(1\) in the category of profinite modules (\cite[Thm. 4.1]{brumerPseudocompactAlgebrasProfinite1966}), \(I\) is projective in said category. Therefore, any surjection \(\FFalgebra{\widehat{F}}^n \to I\) (which is automatically continuous) splits, showing that \(I\) is a direct summand of an abstract free module and thus projective as an abstract \(\FFalgebra{\widehat{F}}\)-module.

By a theorem of Chase (\cite[Thm. 4.32]{rotmanIntroductionHomologicalAlgebra2008}), a ring is semi-hereditary if and only if it is coherent and submodules of flat modules are flat. Moreover, a ring is coherent if and only if any product of flat modules is flat (\cite[Thm. 7.9]{rotmanIntroductionHomologicalAlgebra2008}). Hence, in a semi-hereditary ring, any inverse limit of flat modules is flat. Since \(\FFalgebra{P}\) is an inverse limit of finitely generated free \(\FFalgebra{\widehat{F}}\)-modules, it must be flat over \(\FFalgebra{\widehat{F}}\).

However, flatness is transitive, so \(\FFalgebra{P}\) is flat over \(\Falgebra{F}\). By \cite[Prop. 2.3.21]{cohnFreeIdealRings2006}, it must be the union of free \(\Falgebra{F}\)-modules. Choose now a normal open subgroup \(U\) of \(P\) such that \(U \cap \widehat{G} = \overline{F} \simeq \widehat{F}\), and let \(P_0 =\widehat{G}U \leq_o P\). By the previous remarks, \(\FFalgebra{U}\) is a union of free \(\Falgebra{F}\)-modules, and since tensor products commute with direct sums, \(\FFalgebra{P_0} \simeq \Falgebra{G} \otimes_{\Falgebra{F}} \FFalgebra{U}\) is also a union of free \(\Falgebra{G}\)-modules. As an \(\FFalgebra{P_0}\)-module, \(\FFalgebra{P}\) is isomorphic to \(\FFalgebra{P_0}^{[P\colon P_0]}\). Since the property of being a direct union of free submodules is preserved under direct sums, we are done.
\end{proof}
 
\subsection{Small cancellation groups}

Let
\begin{equation}\label{presentation}
  G = \langle X \mid R \rangle  
\end{equation}
be a finite group presentation where \(R \subseteq F(X)\) is a set of freely reduced and cyclically reduced words in the free group \(F(X)\) such that \(R\) is closed under taking cyclic permutations and inverses. A nontrivial freely reduced word \(u \in F(X)\) is called a \emph{piece} with respect to the presentaion \eqref{presentation} if there exist two distinct elements \(r_1, r_2 \in R\) that have \(u\) as a maximal common initial segment.

Let \(0 < \lambda < 1\). The presentation \eqref{presentation} above is said to satisfy the \(C'(\lambda)\) small cancellation condition if whenever \(u\) is a piece with respect to \eqref{presentation} and \(u \) is a subword of some \(r \in R\), then \[ |u| < \lambda |r|\,.\] Here, \(|v|\) denotes the length of a word \(v\) in the basis \(X\).

Note that if \(G = \langle X \mid S \rangle\) is a group presentation where the set of defining relators \( S \) is not closed for cyclic permutations or inverses, we can always take the set \( R \) which consists of all cyclic permutations of elements of \(S \cup S^{-1}\). If then \(G = \langle X \mid R \rangle\) has property \(C'(1/\lambda)\), then we say that \(G = \langle X \mid S \rangle\) also has property \(C'(1/\lambda)\).

\begin{teo}\label{c16} If \(G =  \langle X \mid S \rangle\) is a finite presentation satisfying the \(C'(1/6)\) small cancelation condition, then \(G\) is residually finite and cohomologically good.
\end{teo}
\begin{proof} By \cite[Thm. 1]{Wi04}, \(G\) acts properly discontinuously and cocompactly on a CAT(0) cube complex \(C\). Observe that \(G\) is word-hyperbolic by Greendlingers' lemma. By \cite[Thm. 1.1]{Ag13}, \(G\) contains a finite index subgroup \(H\) that acts faithfully and specially on \(C\). By \cite[Sec. 6]{Haglund2007}, this implies that \(H\) embeds as a retract of a finite index subgroup of a RAAG, showing that \(H\) and thus \(G\) residually finite. By \cite[Prop. 3.8]{MZ}, we can also conclude that \(G\) is cohomologically good.
\end{proof}

\begin{defn} We say that a collection \(A\) of elements of a group \(G\) is \emph{malnormal} if for any \(a_1\,, a_2 \in A\) such that \({^g\langle a_1 \rangle} \cap \langle a_2 \rangle \neq \{1\}\) we have \(a_1 = a_2\) and \(g \in \langle a_1 \rangle\).
\end{defn}

 The following result is well-known (see, for example, \cite[Thm. 9.10]{Wi12} or \cite[Lem. 5.11]{AM24}).
  
\begin{pro}\label{bigpowers} Let \(F\) be a finitely generated free group and let \(A\) be a finite malnormal collection of elements of \(F\). There exists a constant \(C\) such that for any collection of natural numbers \(\{n_a\}_{a\in A}\) with \(n_a > C\) for all \(a \in A\), the quotient group \[
F/\langle\!\langle a^{n_a} \colon a \in A \rangle\!\rangle\]
satisfies the small cancellation condition \(C'(1/6)\).
\end{pro}

Groups with a \(C'(1/6)\) presentation are known to satisfy many additional properties.

\begin{pro}\label{small} Let \(A = \{a_1, \ldots, a_k\}\) be a malnormal family of elements in a free group \(F\) freely generated by \(x_1,\ldots, x_d\) such that no element of \(A\) is a proper power. Suppose that \[G = \langle x_1,\ldots, x_d \mid a_1^{n_1}, \ldots, a_k^{n_k} \rangle\] is a presentation satisfying the small cancellation condition \(C'(1/6)\). Then:
\begin{enumerate}
    \item \cite[Cor. 2.5]{Ly66} The presentation is combinatorially aspherical: if \(X\) is the 2-complex associated to the presentation and \(\widetilde{X}\) is obtained from its universal covering space by eliminating repeated 2-cells, then \(\widetilde{X}\) is aspherical.
    \item \cite[Thm. 25]{St88} For each \(1 \leq i \leq k\), the order of the image of \(a_i\) in \(G\) is exactly \(n_i\).
\end{enumerate}
\end{pro}

An important consequence of the combinatorial asphericity is that we obtain the following resolution of the trivial \(\bbz[G]\)-module \(\bbz\):
\begin{equation}
    \label{resolution}
0\to \bigoplus_{i=1}^k\bbz[G/\langle a_i\rangle ]\overset{\tau}{\to}\bbz[G]^d\to \bbz[G]\to \bbz \to 0\,,
\end{equation}
where \(\tau\) maps \(x \in \bbz[G/\langle a_i\rangle]\) to \[x\cdot (a_i^{n_i-1} + \cdots +a +1)\cdot \left(\frac{\partial a}{\partial x_1}\,, \cdots \,, \frac{\partial a}{\partial x_d}\right)\,.\]

\section{The profinite completion of finitely generated residually finite groups in \texorpdfstring{\(\mathcal{C}\)}{C}}\label{sect:profcompl}

In this section, \(G = \pi(Y, \G) \) denotes the fundamental group of a finite graph of finitely generated free groups with cyclic edge subgroups and we assume that it is residually finite. We aim to understand the profinite completion of \(G\) and prove that the restriction map in cohomology \(\HH^i_{\mathrm{ct}}(\widehat{G},M) \to \HH^i(G,M)\) is an isomorphism for every finite \(G\)-module \(M\) and every degree \(i \geq 0\).

As explained in \cref{sect:prelim}, we have \(\widehat{G} \cong \hpi(Y, \overline{\G})\). Therefore, our task reduces to understanding \(\overline{\G(v)} \) for each vertex \(v\) of \(Y\) and \(\overline{\G(e)}\) for each vertex \( e \) of \( Y \). If \(G\) is LERF, then \(\overline {\G(x)}\simeq \widehat {\G(x)}\) for all vertices and edges \(x\). If \(G\) is not LERF, this no longer needs to be true. However, we will show that we can still control the structure of \(\widehat{G}\).

First, we show that we can reduce the considerations of the paper to the situation when our graph of groups has only one vertex \(u\), that is, \(G\) is obtained from a base group \(\G(u)\) by finitely many HNN extensions of cyclic subgroups. 

Indeed, choose any vertex \(u\) in \(Y\) and let \(Y_0\) be the new graph obtained by connecting \(u\) to each other vertex \(v\) of \(Y\) by adding a new edge. Extend the graph of groups \(\G\) on \(Y\) to a graph of groups \(\G_0\) on \(Y_0\) by defining \(\G_0(e) = \{1\}\) for each new edge \(e\) added. We clearly have \(\pi(Y_0,\G_0) = G * F_k\) where \(k = |V(Y)-1|\) and \(F_k\) is a free group of rank \(k\), which can be seen by collapsing all the old edges in \(Y\). On the other hand, if \(Y_1\) denotes the graph where we collapsed all the \(k\) new edges, identifying all the vertices of \(Y_0\) with \(u\), then \(Y_1\) is a bouquet with a single vertex \(u\) and edge set in bijection with \(E(Y)\). The graph of groups \(\G_0\) induces a graph of groups \(\G_1\) on \(Y_1\) such that \[\G_1(u) = \underset{v \in V(Y)}{\ast} \G(v)\,.\]

We then have \(\pi(Y_1,\G_1) = G * F_k\) and \(\hpi(Y_1,\overline{\G_1}) = \widehat{G} \amalg \widehat{F_k}\). Both the abstract and the profinite versions of the Mayer-Vietoris sequence imply the existence of a commutative diagram
\[\begin{tikzcd}
{\HH^i_{\mathrm{ct}}(\widehat{G}\amalg \widehat{F_k}, M)} \arrow[r] & {\HH^i(G*F_k,M)}       \\
{\HH^i_{\mathrm{ct}}(\widehat{G},M)} \arrow[u] \arrow[r]            & {\HH^i(G,M)} \arrow[u]
\end{tikzcd}\]
where the vertical arrows are isomorphisms for \(i > 1\). Hence, \(G\) is cohomologically good if and only if \(G * F_k\) is cohomologically good. Therefore, we may assume without loss of generality that \(Y\) has only one vertex \(u\).

Suppose now that the vertex group \(\G(u)\) is free. In order to describe the groups \( \overline{\G(u)} \) and \( \overline{\G(e)} \) for \(e \in E(Y)\), we recall some notation from \cite{AM24}. First, observe that we can present \(G = \pi(Y, \G) \) as
\[G = \left\langle a_1, \ldots, a_k, q_1, \ldots, q_l \,\middle|\, q_j u_j^{m_j} q_j^{-1} = v_j^{n_j} (1\le j\le l)\right\rangle,\]
where:
\begin{enumerate}
    \item \(a_1, \ldots, a_k \) are free generators of \( F = \mathcal{G}(u) \),
    \item  \( m_j, n_j \in \mathbb{Z} \),
    \item  The set \(A= \{ u_j, v_j |1\le j\le l\}\ \)  is a malnormal set of cyclically reduced elements (some of the $u_j$ and $v_i$ may be equal).
\end{enumerate}

We now construct a graph \(\Gamma\) associated to this presentation. The set of vertices \(V(\Gamma)\) is given by the set \( A \), and to each relation \[q_j u_j^{m_j} q_j^{-1} = v_j^{n_j}\] we associate an edge connecting  \(v_j\) and  \(u_j\). The ends of the edges are labeled by nonzero integers: the end corresponding to \(u_j\) is labeled by \(m_j\), and the one corresponding to \(v_j\) is labeled by \(n_j\). These edge-end labels naturally induce labels on the ends of paths in \(\Gamma\): for any path \(\gamma\) in \(\Gamma\), we define the \emph{loop product} \(\operatorname{lp}(\gamma) \in \mathbb{Z} \setminus \{0\}\) as the product of all the labels on the outgoing ends of edges along \(\gamma\).

We denote by \(\overline{\gamma}\) the inverse path of a path \(\gamma\) of \(\Gamma\). A cycle \(\gamma\) in \(\Gamma\) is called \emph{balanced} if \(|\operatorname{lp}(\gamma)| = |\operatorname{lp}(\overline{\gamma})|\); otherwise, it is called \emph{unbalanced}. A connected component \(C\) of \(\Gamma\) is called \emph{clean} if all its cycles are balanced.

The graph \(\Gamma\) encodes all information about Baumslag--Solitar subgroups of \(G\). Since \(G\) is residually finite, the only Baumslag--Solitar subgroups \(BS(n,m)\) that can appear as subgroups of \(G\) are those for which \(|n| = |m|\) or \(\min\{|n|, |m|\} = 1\). For a non-commutative Baumslag--Solitar subgroup \(B\) of \(G\), there exists a unique connected component \(C\) of \(\Gamma\) such that the unbalanced elements of \(B \simeq BS(n,m)\) are conjugate in \(G\) to some element of \(\langle a \rangle\), where \(a \in V(C)\). Moreover, 
the component is non-clean if  $\min\{|n|, |m|\} = 1 $ ($\ne \max\{|n|, |m|\}$). If \(B \cong \mathbb{Z}^2\) is a subgroup of \(G\), then there is also a unique clean connected component \(C\) of \(\Gamma\) such that for every \(a \in V(C)\), some power \(a^k\) is conjugate in \(G\) to an element of \(B\).

The condition that \(G\) is residually finite implies strong consequences for the non-clean components of \( \Gamma \).

\begin{pro}[{\cite[Prop.~4.9]{AM24}}]\label{eulre0} If \(G\) is residually finite, then any non-clean component \(C\) of \(\Gamma\) contains a unique embedded cycle \(\gamma\). Moreover, one of the loop-products \(\operatorname{lp}(\gamma)\) or \(\operatorname{lp}(\overline{\gamma})\) is equal to \( \pm 1 \).
\end{pro}

If \(C\) is a non-clean component and \(\gamma\) is the cycle from the previous proposition, we denote by \(\operatorname{lp}(C)\) the loop-product \(\operatorname{lp}(\gamma)\) or \(\operatorname{lp}(\overline{\gamma})\) that is different from \(\pm 1\). Observe that tehre exists an edge \(e_j\)   of \(Y\) such that \(\iota_{e_j}(\G(e_j)) \) is congugate to \(  \langle u_j^{m_j} \rangle\) in $\G(u)$. 
Let \(C\) be the connected component of \(u_j\) in \(\Gamma\). We define the following sets of prime numbers:
\[P(C) = P(e_j) = \begin{cases}
    \varnothing\,,\text{ if }C\text{ is clean}\,,\\
    \{p \text{ prime}\colon p \mid \operatorname{lp}(C)\}\,,\text{ if }C \text{ is not clean}\,.
\end{cases}\]

If \(P\) is any set of prime numbers, we define \[\widehat{\bbz_P} = \prod_{p \in P} \bbz_p\,.\] The complementary set of primes is \(P^c\), and we have \(\widehat{\bbz} \simeq \widehat{\bbz_P} \times \widehat{\bbz_{P^c}}\). For any element \(g\) of a torsion-free profinite group, one has \(\overline{\langle g \rangle} \simeq \prod_{p \in Q} \bbz_p \leq \widehat{\bbz}\) for some set \(Q\) of primes, and we denote by \(g_P\) the projection of \(g\) to \(\widehat{\bbz_P}\). We are now ready to describe the structure of \(\widehat{G}\).

\begin{teo}\label{profinitecompletion}
Let \(G = \pi(Y, \G)\), where \(Y\) has a unique vertex \(u\), \(F = \G(u)\) is a finitely generated free group, and each \(\G(e)\) is cyclic. Let \(A\) and \(\Gamma\) be as before. 
Then, for each \(e \in E(Y)\), we have \(\overline{\G(e)} \simeq \widehat{\bbz_{P(e)^c}}\), and \(\overline{\G(u)} = \widehat{\G(u)}/N\) where \(N\) is the closed normal subgroup of  \(\widehat{\G(u)}\) generated by \[\{a_{P(C)} : a \in A, \, C \text{ is the component of } a \text{ in } \Gamma\}\,.\]
\end{teo}
\begin{proof} We have a canonical map \(\phi\colon \widehat{F} \to \overline{F}\), and \(N\) is clearly contained in its kernel. In order to prove the theorem, we need to show:
\begin{enumerate}
    \item \(N = \ker \phi\), and
    \item for every \(a \in A\) with connected component \(C\) in \(\Gamma\), the closure of \(\langle a \rangle\) in \(\overline{F}\) is isomorphic to \(\widehat{\bbz_{P(C)^c}}\).
\end{enumerate}

Let \(F_0\) be a finite quotient of \(F\) such that, for each connected component \(C\) of \(\Gamma\), the order of the images of \(a \in C\) is the same and coprime with all primes in \(P(C)\). We can construct a graph of groups \((Y,\G_0)\), where \(\G_0(u) = F_0\) and for each \(e \in E(Y)\), \(\G_0(e)\) is isomorphic to the image of \(\iota_e(\G(e))\) in \(F_0\) (which, by our condition, is also isomorphic to the image of \(t_e(\G(e))\) in \(F_0\)). The group \(\pi(Y,\G_0)\) is a quotient of \(G\), and since \(\pi(Y,\G_0)\) is residually finite as \((Y,\G_0)\) is a graph of finite groups, it follows that \(F_0\) is a quotient of \(\overline{F}\).

Now let \(F_1\) be an arbitrary finite quotient of \(\widehat{F}/N\). By \cref{bigpowers}, we can find a collection \(\{m_C\}\) where \(C\) runs through the connected components of \(\Gamma\) such that:
\begin{enumerate}
    \item the order of the image of each \(a\in A\) in \(F_1\) divides \(m_a =m_C\), where \(C\) is the connected component containing \(a\);
    \item the group presentation \(F_2 = F/\langle\!\langle a^{m_a} \colon a \in A \rangle\!\rangle\), satisfies the \(C'(1/6)\) small cancellation condition.
\end{enumerate}
Observe that by \cref{c16}, \(F_2\) is residually finite.

By \cref{small}, the order of the image of \(a\) in \(F_2\) is exactly \(m_a\). Since \(F_2\) is residually finite, there exists a finite quotient \(F_3\) of \(F\) such that \(F_1\) is a quotient of \(F_3\) and for any \(a \in A\) the order of \(a\) in \(F_3\) is precisely \(m_a\). From the earlier argument, \(F_3\) is a quotient of \(\overline{F}\), and so \(F_1\) is also a quotient of \(\overline{F}\). We conclude that \(\overline{F} \simeq \widehat{F}/N\). Observe that our previous argument also shows that, for each connected component \(C\) of \(\Gamma\) and each \(a\in C\), the closure of \(\langle a \rangle\) in \(\overline{F}\) is isomorphic to \(\widehat{\bbz_{P(C)^c}}\).
\end{proof}

In the next proposition, we describe some properties of the group \(\overline{\G(u)}\) from the previous theorem. For \(a \in F\), we recall that the Fox derivative \(\frac{\partial a}{\partial x_i} \in \bbz[F]\) is the unique element satisfying \[a - 1 = \frac{\partial a}{x_1}(x_1-1) + \cdots + \frac{\partial a}{\partial x_d}(x_d - 1)\,.\]

\begin{pro}\label{profpresent} Let \(F\) be a free group on free generators \(x_1\,,\ldots\,,x_d\), let \(A\) be a finite malnormal collection of elements in \(F\) and let \(\{P(a)\colon a \in A\}\) be a collection of proper subsets of the prime numbers. Let \(G = \widehat{F}/N\), where \(N\) is the closed normal subgroup of \(\widehat{F}\) generated by the elements \(a_{P(a)}\) for \(a \in A\) and define \(A_p = \{a \in A \colon p \in P(a)\}\).

Then, the closed subgroup generated by \(a\) in \(G\) is isomorphic to \(\widehat{\bbz_{P(a)^c}}\). Moreover, for each prime \(p\), we have the following exact sequence:
\begin{equation}\label{asphericity}
0 \to \bigoplus_{a \in A_p} \FFalgebra{G/\overline{\langle a \rangle}} \overset{\tau_G}{\to} \FFalgebra{G}^d \to I_{\FFalgebra{G}} \to 0
\end{equation}
where \[\tau_G(x) = x\cdot \mathrm{tr}_{\overline{a}} \cdot \left(\frac{\partial a}{\partial x_1}\,,\cdots\,,\frac{\partial a}{\partial x_d}\right)\] for \(x \in \FFalgebra{G/\overline{\langle a \rangle}}\) and \(\mathrm{tr}_{\overline{\langle a \rangle}}\) the trace element of \(\FFalgebra{\overline{\langle a \rangle}}^{\overline{\langle a \rangle}}\). In particular, for each prime \(p\), the following holds:
\begin{enumerate}[(i)]
    \item \(\cd_p(G) \leq 2\);
    \item \(\cd_p(G) = 1\) for all primes \(p\) that are not contained in any \(P(a)\);
    \item \(\dim_{\bbf_p} \HH^1_{\mathrm{ct}}(G,\bbf_p) - \dim_{\bbf_p} \HH^2_{\mathrm{ct}}(G,\bbf_p) = d - |A_p|\).
\end{enumerate}
\end{pro}

\begin{proof} The fact that the closed subgroup generated by \(a \in A\) in \(G\) is isomorphic to \(\widehat{\bbz_{P(a)^c}}\) can be established as in the proof of the previous theorem. Let \(n_{a,i}\) be the largest natural number dividing \(i!\) that is not divisible by any prime in \(P(a)\). Define
\[\Gamma_i = \langle x_1\,,\ldots\,, x_d \mid a^{n_{a,i}} = 1\,, a \in A\rangle\] and \(G_i = \widehat{\Gamma}_i\).

Observe that \(G \simeq \varprojlim G_i\). By \cref{bigpowers}, when \(j\) is sufficiently large, the presentation of the group \(\Gamma_j \) satisfies the \(C'(1/6)\) condition, and so, by \cref{c16}, \(\Gamma_j\) is residually finite and cohomologically good. Thus, using \cref{small}, for all sufficiently large \(i > j\), we obtain the commutative diagram
\[\begin{tikzcd}
0 \arrow[r] & \bigoplus_{a \in A} \Falgebra{\Gamma_i/\langle a \rangle} \arrow[d, "{\gamma_{i,j}}"'] \arrow[r, "\tau_i"] & \Falgebra{\Gamma_i}^d \arrow[d] \arrow[r] & \Falgebra{\Gamma_i} \arrow[d] \arrow[r] & \bbf_p \arrow[d, "\simeq"] \arrow[r] & 0 \\
0 \arrow[r] & \bigoplus_{a \in A} \Falgebra{\Gamma_j/\langle a \rangle} \arrow[r, "\tau_j"']                              & \Falgebra{\Gamma_j}^d \arrow[r]           & \Falgebra{\Gamma_j} \arrow[r]           & \bbf_p \arrow[r]                     & 0
\end{tikzcd}
\]
where the maps \(\tau_i\), \(\tau_j\) and \(\gamma_{i,j}\) can be described as follows: \(\tau_i\) sends an element \(x \in \Falgebra{\Gamma_i/\langle a \rangle}\) to \[x\cdot (a^{n_{a,i}-1} + \cdots + 1)\cdot \left(\frac{\partial a}{\partial x_1}\,,\ldots\,, \frac{\partial a}{\partial x_d}\right)\] and \(\gamma_{i,j}\) sends \(x\) to \begin{equation}\label{eqgamma} x\cdot \frac{a^{n_{a,i}}-1}{a^{n_{a,j}}-1} = \frac{n_{a,i}}{n_{a,j}}\cdot x \in \Falgebra{\Gamma_j/\langle a \rangle}\,.\end{equation}

Since the \(\Gamma_i\) and \(\Gamma_j\) are cohomologically good and the modules \(\Falgebra{\Gamma_i/\langle a \rangle}\) and \(\Falgebra{\Gamma_j/\langle a \rangle}\) admit a resolution by finitely generated free modules, by \cite[Thm. 7.3.12]{Wilkes} we have the following commutative diagram:
\[
\begin{tikzcd}
0 \arrow[r] & \bigoplus_{a \in A} \FFalgebra{G_i/\overline{\langle a \rangle|}} \arrow[d, "{\gamma_{i,j}}"'] \arrow[r, "\tau_i"] & \FFalgebra{G_i}^d \arrow[d] \arrow[r] & \FFalgebra{G_i} \arrow[d] \arrow[r] & \bbf_p \arrow[d, "\simeq"] \arrow[r] & 0 \\
0 \arrow[r] & \bigoplus_{a \in A} \FFalgebra{G_j/\overline{\langle a \rangle}} \arrow[r, "\tau_j"']                              & \FFalgebra{G_j}^d \arrow[r]           & \FFalgebra{G_j} \arrow[r]           & \bbf_p \arrow[r]                     & 0
\end{tikzcd}
\]

From the maps $\gamma_{i,j}$, we can construct the $\FFalgebra{G}$-module \(\varprojlim \FFalgebra{G_i/\overline{\langle a \rangle}}\), and by \cref{eqgamma} the resulting module is \(\{0\}\) if \(p \not\in P(a)\) and is isomorphic to \(\FFalgebra{G/\overline{\langle a \rangle}}\) otherwise. Moreover, when \(p \in P(a)\), we can twist the maps \(\tau_i\) by a factor of \(n_{a,i}^{-1} \pmod p\) and \(\gamma_{i,j}\) by a factor of \(n_{a,j}\cdot n_{a,i}^{-1} \pmod{p}\) on each summand \(\Falgebra{\Gamma_i/\langle a\rangle}\) to ensure they satisfy \[\tau_i(x) = x\cdot \mathrm{tr}_{\overline{\langle a \rangle}} \cdot \left(\frac{\partial a}{\partial x_1}\,,\cdots\,,\frac{\partial a}{\partial x_d}\right)\,.\] This prove the existence of \cref{asphericity}.

Consider now \(a\) as an element of \(G\). If \(p \in P(a)\), then since \(\overline{\langle a \rangle} \simeq \widehat{\bbz_{P(a)^c}}\), \(\bbf_p\) is a projective profinite \(\FFalgebra{\overline{\langle a \rangle}}\)-module, and hence \[\FFalgebra{G/\overline{\langle  a\rangle}} \simeq \FFalgebra{G} \widehat{\otimes}_{\FFalgebra{\overline{\langle a \rangle}}} \bbf_p = \widehat{^G \bbf_p}\] is a projective profinite \(\FFalgebra{G}\)-module. Therefore, \(\cd_p(G) \leq 2\). The statement (ii) is clear, and in order to show (iii) we apply \(\Hom_{\FFalgebra{G}}(-,\bbf_p)\) to \cref{asphericity} to obtain the exact sequence \[0 \to \HH^1_{\mathrm{ct}}(G,\bbf_p) \to \bbf_p^d \to \bigoplus_{\substack{a \in A\\ p \in P(A)}} \bbf_p \to \HH^2_{\mathrm{ct}}(G,\bbf_p) \to 0\,.\] This implies the desired equality.
\end{proof}

The previous theorem suggests the natural question of whether, for a given finite malnormal family \(A\subset F\) and a given collection of proper subsets of primes \(\{P(a)\}_{a \in A}\), we can decide if \(\cd_p(G)\) is \(2\) or less than \(2\). Consider, for example, the situation when \(A\) consists of a single element \(a\) in the free group \(F = \langle x_1\,,\ldots\,, x_d \mid \varnothing\rangle\) and \(P\) is some proper subset of the prime numbers. We put \[G = \widehat{F}/\langle\!\langle a_P\rangle\!\rangle\,,\] and we consider the following statements where each one implies the next:
\begin{enumerate}[(a)]
    \item The element \(a\) is primitive in \(F\).
    \item There exists a free profinite subgroup \(H\) of \(G\) of rank \(d-1\) such that \[G = H \amalg \overline{\langle a \rangle}\].
    \item \(G\) is projective.
\end{enumerate}

The condition \textup{(b)} has the following equivalent formulation.

\begin{pro}
The condition \textup{(b)} is equivalent to:
\begin{enumerate}
    \item[(b')] For any finite group \(M\) and any element \(g \in M\) with \(o(g)\) coprime with all primes in \(P\), the number of homomorphisms \(\varphi \colon G \to M\) such that \(\varphi(a) = g\) is exactly \(|M|^{d-1}\).
\end{enumerate}
\end{pro}
\begin{proof} It is clear that (b) implies (b'), so assume (b') holds. It was shown in \cite[Prop. 3.1]{garridoFreeFactorsProfinite2023} that \(\overline{\langle a \rangle}\) is a profinite free factor of \(G\) if and only if for every finite group \(M\) and any homomorphism \(\gamma\colon \overline{\langle a \rangle} \to M\) the number of extensions \(\widetilde{\gamma}\colon G \to M\) is independent of \(\gamma\). By (b'), this number is constant and equal to \(|M|^{d-1}\), so there exists a closed subgroup \(H\) of \(G\) such that \(G = H \amalg \overline{\langle a \rangle}\) and \[\left |\mathrm{Hom}(H,M)\right | = |M|^{d-1}\] for every finite group \(M\). We must show that this implies that \(H\) is a free profinite group of rank \(d-1\).

Let \(\mathrm{Epi}(H,M)\) denote the set of epimorphisms from \(H\) onto a finite group \(M\), and let \(L\) be the free profinite group on \(d-1\) generators. Observe that \(|\mathrm{Hom}(L,M)| = |\mathrm{Hom}(H,M)| = |M|^{d-1}\). We claim that the equality \(|\mathrm{Epi}(H,M)| = |\mathrm{Epi}(L,M)|\) also holds for every finite group \(M\). Indeed, we have the identity \[\mathrm{Hom}(H,M) = \bigcup_{N \leq M} \mathrm{Epi}(H,N)\,,\] so proceeding by induction on \(|M|\) we obtain the desired equality. This implies that the set of finite quotients of \(H\) and \(L\) coincide. Since both are finitely generated profinite groups (with \(H\) being a quotient of \(G\)), they must be isomorphic by a well known result of Dixon, Formanek, Poland and Ribes (see \cite[Thm. 3.2.3]{Wilkes}).
\end{proof}

We propose the following conjecture.

\begin{Conj}[Baby Remeslennikov Conjecture]   
\label{Baby}
The implication \textup{(c)} $\Rightarrow$ \textup{(a)} holds.
\end{Conj}

The relation with the Remeslennikov problem is as follows: \textup{(b)} and a positive solution of Remeslennikov's question in the class of one-relator groups implies \textup{(a)}. Indeed, consider the group \(\Gamma_k = F/\langle\!\langle a^k\rangle\!\rangle\) for \(k > 1\) coprime with all primes in \(P\). By the Newman Spelling Theorem, it is hyperbolic, and by \cite{Wi21} it is virtually compact special. In particular, \(\Gamma_k\) is residually finite and cohomologically good. If \textup{(b)} holds, then \(\widehat{\Gamma_k}\) is virtually free profinite: the closed normal subgroup \(N\) generated by \(H\) is free profinite. Moreover, the Reidemeister-Schreier rewritting process with respect to the transversal \(\{1,a,\ldots,a^{k-1}\}\) shows that \(U = N \cap \Gamma_k\) is a one-relator group whose profinite completion is isomorphic to \(N\). Thus, if the free group is rigid in the class of one-relator groups, \(U\) is free and so \(\Gamma_k\) is virtually free. By \cite{garridoFreeFactorsProfinite2023}, \(\langle a \rangle\) is a free factor of \(\Gamma_k\). Therefore, by \cite[Prop. III.3.7]{LSbook}, \(a\) is primitive.

A possible approach to proving \cref{Baby} consists in extending the work of Wilton~\cite{Wi18}. Assume that \(a\) is not primitive.
 Without loss of generality, we may assume that $a$ is not contained in a proper free factor of \(F\). Let us assume that \(k \geq |a|\) is greater than or equal to the word length of \(a\). Then, by \cite[Thm 5.7]{HW01}, \(\Gamma_k\) is locally quasiconvex. Moreover, by the last corollary of \cite{FKS72} and \cite[Prop. III.3.7]{LSbook} the groups \(\Gamma_k\) are one-ended. The last step would be to show that for some \(k \geq |a|\) and coprime with all primes in \(P\) the group \(\Gamma_k\) contains a surface subgroup. Since \(\Gamma_k\) is locally quasiconvex and virtually compact special, it would virtually retract onto this surface subgroup, producing a non-trivial second cohomology class with coefficients in some finite \(\Falgebra{\Gamma_k}\)-module for some \(p \in P\). However, the following proposition would contradict the projectivity of \(G\):

\begin{pro}\label{powerquotient}
Assume that $G$ is projective. Then \[\cd_p\!\left(\,\widehat{\Gamma_k}\,\right) \leq 1\] for every \(k\) coprime with all primes in \(P\) and \(p \in P\).
\end{pro}
\begin{proof} Observe that \(\widehat{\Gamma_k}\) is a quotient of \(G\). As in the proof of Proposition~\ref{profpresent}, we have a commutative diagram
\[\begin{tikzcd}
0 \arrow[r] & \FFalgebra{G/\overline{\langle a \rangle}} \arrow[r, "\tau_G"] \arrow[d, "\gamma"']                 & \FFalgebra{G}^d \arrow[r] \arrow[d]        & \FFalgebra{G} \arrow[r] \arrow[d]        & \bbf_p \arrow[r] \arrow[d] & 0 \\
0 \arrow[r] & \FFalgebra{\widehat{\Gamma_k}/\overline{\langle a \rangle}} \arrow[r, "\tau_{\widehat{\Gamma_k}}"'] & \FFalgebra{\widehat{\Gamma_k}}^d \arrow[r] & \FFalgebra{\widehat{\Gamma_k}} \arrow[r] & \bbf_p \arrow[r]           & 0
\end{tikzcd}\] of \(\FFalgebra{G}\)-modules with exact rows where \(\gamma\) is induced by the quotient map \(G \to \widehat{\Gamma_k}\). Moreover, the top (resp. bottom) row is a projective resolution of \(\bbf_p\) over \(\FFalgebra{G}\) (resp. \(\FFalgebra{\widehat{\Gamma_k}}\)) since \(|\overline{\langle a \rangle}|\) is coprime to \(p\).

We recall how the maps \(\tau_G\) and \(\tau_{\widehat{\Gamma_k}}\) are defined: if we write \[a-1 = \frac{\partial a}{\partial x_1}(x_1-1)+\cdots+\frac{\partial a}{\partial x_d}(x_d-d)\] in \(\bbz[F]\), then both are given by \[x \cdot \mathrm{tr}_{\overline{\langle a \rangle}} \left(\frac{\partial a}{\partial x_1}\,,\ldots\,, \frac{\partial a}{\partial x_d}\right)\] for \(x\) in \(\FFalgebra{G}\) or \(\FFalgebra{\widehat{\Gamma_k}}\) with the trace element \(\mathrm{tr}_{\overline{\langle a \rangle}}\) computed accordingly.

Finally, observe that for any finite \(p\)-primary \(\widehat{\Gamma_k}\)-module \(M\) we have isomorphisms \[M^{\overline{\langle a \rangle}} \simeq\Hom_{\widehat{\Gamma_k}}(\FFalgebra{\widehat{\Gamma_k}/\overline{\langle a \rangle}}, M) \simeq \Hom_G(\FFalgebra{\widehat{\Gamma_k}/\overline{\langle a \rangle}},M) \overset{\Hom(\gamma)}{\simeq} \Hom_G(\FFalgebra{G/\overline{\langle a \rangle}},M)\,,\]
\[M^d \simeq \Hom_{\widehat{\Gamma_k}}(\FFalgebra{\widehat{\Gamma_k}}^d,M) \simeq \Hom_G(\FFalgebra{\widehat{\Gamma_k}}^d,M) \simeq \Hom_G(\FFalgebra{G}^d,M)\,.\] Therefore, applying \(\Hom_{\FFalgebra{G}}(-,M)\) to the diagram above we get the commutative diagram \[\begin{tikzcd}
M^d \arrow[r, "\Hom(\tau_G)"]                                         & M^{\overline{\langle a \rangle}} \arrow[r]                      & 0 \\
M^d \arrow[r, "\Hom(\tau_{\widehat{\Gamma_k}})"'] \arrow[u, "\simeq"] & M^{\overline{\langle a \rangle}} \arrow[r] \arrow[u, "\simeq"'] & 0
\end{tikzcd}\]
This gives us the isomorphism \[\HH^2_{\mathrm{ct}}(\widehat{\Gamma_k},M) \simeq \mathrm{coker}(\Hom(\tau_{\widehat{\Gamma_k}})) \simeq \mathrm{coker}(\Hom(\tau_G)) \simeq \HH^2_{\mathrm{ct}}(G,M)\] for all finite \(p\)-primary \(\widehat{\Gamma_k}\)-modules \(M\). Since \(G\) is projective, \(\mathrm{cd}_p(G) \leq 1\) and thus \(\HH^2_{\mathrm{ct}}(G,M) = 0\), showing that \(\mathrm{cd}_p(\widehat{\Gamma_k}) \leq 1\) as well.
\end{proof}
 
We finish this section by proving \cref{goodness}.

\begin{proof}[Proof of \cref{goodness}] We may assume that \(G\) is not virtually free. Moreover, upon passing to a subgroup of finite index, we may assume that \(G\) is torsion-free. Since \(G\) has cohomological dimension \(2\), it suffices to show that for every finite-index subgroup \(H \leq G\) and any prime \(p\) the canonical map \[\HH_{\mathrm{ct}}^2(\widehat{H}, \bbf_p) \to \HH^2(H,\bbf_p)\] is an isomorphism. Observe that for finitely generated \(H\) this map is always injective, and so we only have to check that the dimensions of both spaces coincide.

By replacing \(H\) with \(H*F_k\) for some finitely generated free group \(F_k\), we may assume that \(H = \pi(Y,\G)\) where \(Y\) is a finite graph with a single vertex \(u\), \(F = \G(u)\) is a finitely generated free group of rank \(d\) and each edge group \(\G(e)\) is cyclic.

From the Mayer-Vietoris sequence
\[0 \to \bbf_p^{|E(Y)|} \to \HH^1(H,\bbf_p) \to \bbf_p^d \to \bbf_p^{|E(Y)|} \to \HH^2(H,\bbf_p) \to 0\]
for the fundamental group of a graph of groups we conclude that \[\dim_{\bbf_p} \HH^2(H,\bbf_p) = \dim_{\bbf_p} \HH^1(H,\bbf_p) - d\,.\] The Mayer-Vietoris sequence for the profinite fundamental group of a graph of groups gives us
\begin{equation}\label{conth2}
    \begin{split}
        \dim_{\bbf_p} \HH^2_{\mathrm{ct}}(\widehat{H},\bbf_p) &= \dim_{\bbf_p} \HH^1_{\mathrm{ct}}(\widehat{H},\bbf_p) - |E(Y)| + \sum_{e \in E} \dim_{\bbf_p} \HH^1_{\mathrm{ct}}(\overline{\G(e)},\bbf_p)\\
        &\hphantom{aaaa} -\dim_{\bbf_p} \HH^1_{\mathrm{ct}}(\overline{\G(u)},\bbf_p) + \dim_{\bbf_p} \HH^2_{\mathrm{ct}}(\overline{\G(u)},\bbf_p)\,.
    \end{split}
\end{equation}

Let \(\Gamma\) be the graph associated with the standard presentation of \(\pi(Y,\G)\), and denote by \(\Gamma_p\) the union of the components \(C\) of the graph \(\Gamma\) for which \(p \in P(C)\). Put \(A_p = V(\Gamma_p)\) and \(E_p = E(\Gamma_p)\). By \cref{eulre0}, the Euler characteristic of such a component is \(0\) and so \(|A_p| = |E_p|\). Observe now that \(|E(Y)| = |E(\Gamma)|\) and so \[\sum_{e \in E(Y)} \dim_{\bbf_p} \HH^1_{\mathrm{ct}}(\overline{\G(e)},\bbf_p) - |E(Y)| = -|E_p|\,.\] From \cref{profpresent} we have \[\dim_{\bbf_p} \HH^1_{\mathrm{ct}}(\overline{\G(u)},\bbf_p) - \dim_{\bbf_p} \HH^2_{\mathrm{ct}}(\overline{\G(u)},\bbf_p) = d - |A_p|\,,\] and therefore from \cref{conth2} we conclude that 
\begin{align*}
    \dim_{\bbf_p} \HH^2_{\mathrm{ct}}(\widehat{H},\bbf_p) &= \dim_{\bbf_p} \HH^1_{\mathrm{ct}}(\widehat{H},\bbf_p) - d\\
    &= \dim_{\bbf_p} \HH^1(H,\bbf_p) - d\\
    &= \dim_{\bbf_p} \HH^2(H,\bbf_p)\,.
\end{align*}
This finishes the proof.
\end{proof}

\section{Relative splitting theorem for virtually free groups}\label{sect:relative}

Our goal in this section is to prove a relative version of the splitting theorem for finitely generated virtually free groups. The following theorem implies \cref{thm-relative-splitting-virtually-free-groups}.

\begin{teo}\label{relative splitting} Let \(G\) be a finitely generated virtually free group and let \(C_1\,,\ldots\,,C_n\) be subgroups of \(G\). Suppose that \(\widehat{G} \leq P = K\amalg L\), and that each \(C_i\) is conjugate in \(P\) to a subgroup of either \(K\) or \(L\). Then \(G\) splits as a free product \[G \simeq \left(\underset{i \in I}{\ast} G_i\right)  *F\,,\] where \(F\) is a free subgroup of \(G\), each \(G_i\) is contained in some \(P\)-conjugate of \(K\) or \(L\) and each \(C_j\) is contained in a \(G\)-conjugate of some \(G_i\).
\end{teo}
\begin{proof} If \(G\) is contained in a \(P\)-conjugate of \(K\) or \(L\), then there is nothing to prove, so we assume that this is not the case. Let \(p\) be a prime. By \cref{elliptic}(i), we have \[I_{\FFalgebra{P}} \simeq \widehat{^PI_{\FFalgebra{K}}} \oplus \widehat{^PI_{\FFalgebra{L}}}\,,\] where this isomorphism is induced by the natural inclusions
\[I_{\FFalgebra{K}}\,, I_{\FFalgebra{L}} \hookrightarrow I_{\FFalgebra{P}}\,.\]
In particular, we obtain a derivation \(d_L\colon P \to \FFalgebra{P}\) defined by \(g \mapsto g-1\) in \(I_{\FFalgebra{P}}\) and then taking the quotient modulo \(\widehat{^PI_{\FFalgebra{L}}}\).

Let \(U\) be a subgroup of \(G\) contained in a \(P\)-conjugate of \(K\) or \(L\). Then, by \cref{claim-conjugate-derivations}, the restriction of \(d_L\) to \(U\) is inner. In particular, this applies to  each  subgroup  \(U = C_i\).

\begin{claim} There is a finitely generated \(\Falgebra{G}\)-submodule \(N\) of \(\FFalgebra{P}\) containing \(d_L(G)\) and such that the induced derivation \(d_L\colon  G \to N\) is inner when restricted to the subgroups \(C_i\).
\end{claim}
\begin{proof}
Since \(G\) is finitely generated, say by \(x_1\,,\ldots\,,x_k\), then \(d_L(G)\) is contained in the \(G\)-submodule generated by \(d(x_1)\,,\ldots\,,d(x_k)\). Moreover, for each \(i\) there is also an element \(a_i\) in \(\FFalgebra{P}\) such that \[d_L(g) = (g-1)a_i\] for all \(g \in C_i\). Hence, we may take \(N\) to be the \(G\)-submodule generated by all the \(d(x_i)\), \(a_i\).
\end{proof}

By  \cref{flatness}, the \(\Falgebra{G}\)-module \(N\) is contained in a finitely generated free \(\Falgebra{G}\)-submodule \(M\) of \(\FFalgebra{P}\). Thus we view \(d_L\) as a derivation \(d\colon G \to M< \FFalgebra{P}\).

\begin{claim}\label{claim-d-not-inner} \(d\) is not inner.
\end{claim} 
\begin{proof} 
Assume that \(d\) is inner. Let \(a \in \FFalgebra{P}\) be such that \(d_L(g) = (g-1)a\) for all \(g \in G\). By the definition of \(d_L\), we have:
\[(g-1)a \equiv 0 \pmod{\widehat{^{P}I_{\FFalgebra{L}}}}\,,\]
\[(g-1)a \equiv g-1 \pmod{\widehat{^{P}I_{\FFalgebra{K}}}}\] for all \(g \in \widehat{G}\), meaning that the image of \(a\) is a \(\widehat{G}\)-fixed point in \(\FFalgebra{P/L}\) and the image of \(1-a\) is a \(\widehat{G}\)-fixed point in \(\FFalgebra{P/K}\). However, \(\FFalgebra{P/L}^{\widehat{G}}= \FFalgebra{P/K}^{\widehat{G}} = 0\),   by \cref{fixedpoints}. Hence, \(a \in \widehat{^{P}I_{\FFalgebra{L}}}\) and \(1-a \in \widehat{^{P}I_{\FFalgebra{K}}}\), from which we obtain \[1 = (1-a)+a \in \widehat{^{P}I_{\FFalgebra{K}}} \oplus \widehat{^{P}I_{\FFalgebra{L}}} = I_{\FFalgebra{P}}\,,\] a contradiction.
\end{proof}

Since \(d\) is a non-inner derivation to a projective \(\Falgebra{G}\)-module, by \cref{DD} there exists a \(G\)-tree \(X''\) with finite edge stabilizers such that a subgroup \(H\) stabilizes a vertex of \(X''\) if and only if the restriction of \(d\) to \(H\) is inner.  
Let \(v_1\,, \ldots\,, v_n\) be vertices of \(X''\) fixed by \(C_1\,, \ldots\,, C_n\) respectively. Since \(G\) is finitely generated, there exists a \(G\)-subtree \(X' \subseteq X''\) containing \(v_1\,, \ldots\,, v_n\) such that \(G / X'\) is finite. After collapsing all compressible edges, we may assume that \(X'\) is incompressible.  Observe that each \(C_i\) fixes at least one vertex in \(X'\).

\begin{claim}\label{p-part-infinite} For any \(g \in P\) and any vertex \(v \in X'\) with \(G_v\) infinite, the indexes \([\overline{G_v}\colon \overline{G_v}\cap {^gK}]\) and \([\overline{G_v}\colon \overline{G_v}\cap {^gL}]\) are either \(1\) or \(\prod_{p\text{ prime}} p^\infty\).
\end{claim}
\begin{proof}
It suffices to prove the claim for \(K\), as the same argument works for \(^gK\) or \(^gL\). If \(G_v \leq {K}\), then \(\overline {G_v} \leq {K}\) and hence the index is \(1\). Assume now that \(G_v\) is not contained in \(K\) and consider a finite index torsion-free normal subgroup \(H_v\) of \(G_v\). We cannot have \(H_v \leq {K}\) either, since in this case \(N_G(H_v)\leq {K}\) by \cref{normalizers} and this normalizer contains  \(G_v\). If \(x\) is any element of \(H_v \smallsetminus {K}\), then the infinite cyclic subgroup \(Z\) generated by \(x\) intersects \({K}\) trivially by \cref{normalizers} again.

Since \(G\) is cyclic subgroup separable, we have that \(\overline{Z} = \widehat{\bbz}\) is a subgroup of \(\overline {G_v}\), and moreover, by \cref{elliptic}, \(\overline{Z} \cap {K} = \{1\}\). This shows that the \(p\)-part of the index \([\overline {G_v}\colon \overline {G_v} \cap {K}]\) must be infinite for every prime \(p\).
\end{proof}

\begin{claim} For each \(G_v\) there exists \(g \in P\) such that \(G_v\) is contained in \(^gK\) or \(^g L\).
\end{claim}
\begin{proof} If \(G_v\) is finite, then by \cite[Cor. 7.1.3(a)]{ribesProfiniteGraphsGroups2017}, some conjugate of \(G_v\) in \(\widehat{G}\) lies in \(K\) or \(L\).  Therefore, we may assume that \(G_v\) is infinite. Since \(d\) is inner in \(G_v\), there exists \(a \in \FFalgebra{P}\) such that \(d(h) = (h-1)a\) for every \(h \in G_v\). As in the proof of \cref{claim-d-not-inner}, we have:
\[(h-1)a \equiv 0 \pmod{\widehat{^{P}I_{\FFalgebra{L}}}}\,,\]
\[(h-1)a \equiv h-1 \pmod{\widehat{^{P}I_{\FFalgebra{K}}}}\]
for all \(h \in G_v\), which means that the image of \(a\) is a \(G_v\)-fixed point in \(\FFalgebra{P/L}\) and that the image of \(1-a\) is an \(G_v\)-fixed point in \(\FFalgebra{P/K}\). If both were actually zero in those quotients, it would lead to the same contradiction as in \cref{claim-d-not-inner}. Hence, at least one of \(\FFalgebra{P/L}\) or \(\FFalgebra{P/K}\) must contain a non-zero \(G_v\)-fixed point. By \cref{fixedpoints}, there exists \(g \in P\) such that  either \([\overline{G_v}\colon \overline{G_v} \cap {^gL}|\) or \([\overline{G_v}\colon \overline{G_v}\cap {^gK}]\) is not divided by \(p^{\infty}\). Therefore, by \cref{p-part-infinite}, \(G_v\) is contained in \(^gL\) or \(^gK\).  
\end{proof}

Let \(\mathcal{S}\) be the set of subgroups of \(P\) that are conjugates of \(K\) or \(L\). By \cref{normalizers}, any two distinct subgroups in \(\mathcal{S}\) intersect trivially. Thus, for each vertex \(v\) of \(X'\), there exists a unique subgroup \(U \in \mathcal{S}\) such that \(G_v \leq U\). Moreover, if \(G_e\) is a non-trivial edge stabilizer, then the stabilizers of its adjacent vertices are contained in the same subgroup of \(\mathcal{S}\).

Let \(X\) be the tree obtained from \(X'\) by collapsing all edges with non-trivial stabilizers. Assume, for contradiction, that the result tree consists of a single vertex. Then the closed subgroup of \(\widehat{G}\) generated by \(\{G_v\colon v \in V(X')\}\) is normal and contained in a \(P\)-conjugate of \(K\) or \(L\). Hence, \(\widehat{G}\) itself is contained in the same conjugate of \(K\) or \(L\), which contradicts our assumption. Therefore, \(X\) is not a point.

Let \(v\) be a vertex of \(X\) and let \(U \in \mathcal{S}\) be such that \(G_v \leq U\). If \(g \in G\) stabilizes \(v\), then there are two vertices \(u_1\) and \(u_2 = gu_1\) of \(X'\) such that \(G_{u_1}\) and \(G_{u_2}\) are non-trivial and contained in \(U\). Therefore, \(G_{u_1} \leq {^{g^{-1}}U} \cap U \neq \{1\}\), and so \(g \in U\). Therefore, we have obtained an infinite tree on which \(G\) acts with trivial edge stabilizers, such that each \(C_i\) fixes at least one vertex, and any vertex stabilizer lies in some subgroup from \(\mathcal{S}\). This implies that \(G\) splits as the desired free product.
\end{proof}

\begin{proof}[Proof of \cref{thm-relative-splitting-virtually-free-groups}] Let \(C_1\,,\ldots\,, C_n\) be subgroups of the form \(G \cap {^gL}\) or \(G \cap {^hK}\) with \(g,h \in P\). Then, by \cref{relative splitting}, \(G\) splits as a free product \[G \simeq \left(\underset{i \in I}{\ast} G_i\right) * F\] where \(F\) is a free group, each \(G_i\) is contained in some \(P\)-conjugate of \(K\) or \(L\) and each \(C_j\) is contained in a \(G\)-conjugate of some \(G_i\). It is clear that if \(C_j \leq G_i\), then \(C_j = G_i\). Since the number of possible non-trivial \(G_i\) is bounded, by choosing a sufficiently large number of subgroups \(C_j\), we obtain that all \(G_i\) are of the form \(G \cap {^g L}\) or \(G \cap {^hK}\).
\end{proof}

\section{Detecting free products and free factors}\label{sect:main-thm}

Our first goal in this section is to prove the following theorem, which in turn implies \cref{thm-detects-free-products}.

\begin{teo}\label{main-thm} Let \(G=\pi(Y,\G) \) be the fundamental group of a finite graph of finitely generated virtually  free groups with virtually cyclic edge groups. Assume that \(G\) is LERF and that \(\widehat{G} = L\amalg K\) decomposes non-trivially as a free profinite product. Then, \(G\) decomposes as a free product 
\begin{equation}\label{decomp}
G= \left(\underset{i \in I}{\ast} G_i\right) * F, \end{equation}
where \(F\) is a free group and each \(G_i\) is contained in some \(\widehat{G}\)-conjugate of \(L\) or \(K\). Moreover,  if \(C_1\leq \G(v_1)\,,\ldots\,, C_n\leq \G(v_n)\) are finitely generated subgroups of vertex groups lying in some \(\widehat{G}\)-conjugate of \(L\) or \(K\), then we can also choose the \(G_i\) in \cref{decomp} such that each \(C_j\) lies in some \(G\)-conjugate of some \(G_i\).
\end{teo}
\begin{proof}  Recall that by \cref{lem-residually-finite}, \(G\) contains a finite index subgroup \(N\) that is the fundamental group of a finite graph of finitely generated free groups with cyclic edge groups. By \cref{goodness}, \(N\) is cohomologically good. In particular, the profinite completion of \(N\) is torsion-free, and so the closure \(\overline{N} \simeq \widehat{N}\) of \(N\) in \(\widehat{G}\) is torsion-free.

We call an edge \(e\) of \(Y\) \emph{parabolic} with respect to the splitting \(L \amalg K\) if \(\G(e)\) is infinite and is a subgroup of some \(\widehat{G}\)-conjugate of \(L\) or \(K\). By \cref{elliptic}, if \(\G(e)\) has infinite interserction with a conjugate of \(K\) or \(L\), then \(e\) is parabolic. 

Let \(\Delta\) be the subgraph of \(Y\) obtained by removing all the parabolic edges. Let \((\widetilde{Y}, \widetilde{\G})\) be the quotient graph of groups obtained by collapsing each connected component \(X\) of \(\Delta\) to a point \(x\) and setting \(\widetilde{\G}(x) = \pi(X,\G)\).

\begin{claim}\label{claim-after-collapsing} The vertex groups \(\widetilde{\G}(x)\) are virtually free.
\end{claim}
\begin{proof}
Since \(G\) is LERF, we have \[\widehat{\pi(X, \widehat{\G})} \simeq \widehat{\pi(X, \mathcal{G})} \simeq \overline{\pi(X, \mathcal{G})} \subseteq \widehat{G}\,.\] Let \(U = \overline{N} \cap \overline{\pi(X, \mathcal{G})}\) and for each \(g \in G\) let \(U_{g,L} = U \cap {^g L}\) and \(U_{g,K} = U \cap {^gK}\). We regard the \(U_{g,L}\) and \(U_{g,K}\) as subgroups of \(\pi(X,\widehat{\G})\). They have trivial intersection with the edge subgroups of \((X,\widehat{\G})\), and thus by \cref{projectiveintersection} each one of the \(U_{g,L}\) and \(U_{g,K}\) are projective. Applying \cref{projectiveintersection} again, we deduce that \(U\) is projective. By \cref{BSsubgroups}, \(\pi(X,\G)\) contains no Baumslag--Solitar subgroups. Therefore, by \cref{wilton}, \(\pi(X,\G)\) is virtually free.
\end{proof}

From now on, we assume that all edge subgroups \(\G(e)\), \(e \in E(Y)\), are infinite and contained in some \(\widehat{G}\)-conjugate of \(K\) or \(L\). Applying \cref{thm-relative-splitting-virtually-free-groups}, we can decompose each vertex group \(\G(v)\) for \(v \in V(Y)\) as \[\G(v) = \left(\underset{i \in I_V}{\ast} H_{i,v}\right) * F_v\,,\] where \(F_v\) is a free subgroup of \(\G(v)\) and each \(H_{i,v}\) is of the form \(\G(v) \cap {^gL}\) or \(\G(v) \cap {^gL}\) for some \(g \in \widehat{G}\). In particular, if \(C_j \leq \G(v)\), then it is \(\G(v)\)-conjugate to a subgroup of some \(H_{i,v}\). Also, if \(\iota(e) = v\), then \(\G(e)\) is \(\G(v)\)-conjugate to a subgroup of some \(H_{i,v}\), and the same holds for \({^{q_e}\G(e)}\) if \(t(e) = v\). By adjusting the \(q_e\) if necessary, we can assume that \(\G(e)\) (respectively, \({^{q_e}\G(e)}\)) belongs to some \(H_{i,v}\) if \(v = \iota(e)\) (respectively, \(v = t(e)\)).

This allows us to construct another graph of groups \((Z,\mathcal{H})\) by replacing each vertex \(v \in V(Y)\) with a line having \(|I_v|+1\) vertices, whose vertex groups are the \(H_{i,v}\) and \(F_v\), with all edge groups in this line being trivial. For each \(e \in E(Y) \subset E(Z)\), we set \(\mathcal{H}(e) = \G(e)\), and these are all non-trivial.

Let \(\Sigma\) be the subgraph of \(Z\) obtained by removing all the edges in \(E(Y)\), and let \((\widetilde{Z}, \widetilde{\mathcal{H}})\) be the quotient graph of groups obtained by collapsing each connected component \(X\) of \(\Sigma\) to a point \(z\) and setting \(\widetilde{\mathcal{H}}(z) = \pi(X,\G)\). Then all edge groups of \(\widetilde{Z}\) are trivial, and observe that all \(C_i\) are conjugate to some vertex group of \(\widetilde{Z}\).

\begin{claim} For each vertex \(v\) in \(\widetilde{Z}\), there exists \(g \in \widehat{G}\) such that \(\widetilde{\mathcal{H}}(v)\) is contained in \(^gK\) or \(^gL\).
\end{claim}
\begin{proof}
Let \(\mathcal{S}\) be the set of subgroups of \(\widehat{G}\) that are conjugates of \(K\) or \(L\). Recall that any two distinct subgroups in \(\mathcal{S}\) intersect trivially by \cref{normalizers}, and that the normalizer in \(\widehat{G}\) of any subgroup of \(\mathcal{S}\) coincides with the subgroup itself. We have \(\widetilde{\mathcal{H}}(v) \simeq \pi(X,\G)\). Each \(\mathcal{H}(u)\), \(u \in V(X)\), lies in a unique subgroup \(U_u \in \mathcal{S}\), and since all edge groups \(\mathcal{H}(e)\), \(e \in E(X)\), are non-trivial, we have \(U_u = U\) for all \(u \in V(X)\). Thus, for any edge \(e \in E(X)\), \(\mathcal{H}(e) \subseteq U \cap {^{q_e}U}\), and so \(q_e \in U\). Therefore, \(\widetilde{\mathcal{H}}(v) \leq U\).
\end{proof}

Since all edge groups of \((\widetilde{Z}, \widetilde{\mathcal{H}})\) are trivial, the theorem follows.
\end{proof}

Now we prove both \cref{cor-relative-splitting} and \cref{oneendedH}.

\begin{proof}[Proof of \cref{cor-relative-splitting}] Observe that  \(G\) is LERF by \cite[Thm. 1]{appel_subgroup_1984}. By \cref{claim-after-collapsing}, if \(c_1\) is not contained in any \(\widehat{G}\)-conjugate of \(L\) or \(K\), then \(G\) is free. In this case, by \cite[Prop. II.5.13]{LSbook}, without loss of generality we can assume that \(\langle c_1\rangle\) is a free factor of \(F_1\), so the corollary follows. If \(G\) is not free, then again without loss of generality we can assume that \(c_1 \in L\). By \cref{main-thm}, we have a non-trivial splitting \[G = H * K\] where \(c_1 \in H \leq L\).  If both \(F_1\) and \(F_2\) were contained in \(H\), then we would have \(K = \{1\}\), a contradiction. Therefore, we may assume that \(F_1 \not\leq H\), and the Kurosh subgroup theorem gives us a non-trivial splitting \(F_1 = H_1 * K_1\) where \(c_1 \in H_1\).
\end{proof}

\begin{proof}[Proof of \cref{oneendedH}] Assume in \cref{main-thm} that \(L = \overline{H}\). If \(H\) is one-ended, then \(H\) belongs to some \({^g G_i}\). This subgroup \({^g G_i}\) must lie in \(\overline{H} \), and since \(G\) is LERF, we conclude that \({^g G_i} = H\). If \(H\) is contained in a conjugate of a vertex subgroup, we apply \cref{main-thm} with \(C_1 = H\). Then again \(H\) belongs to some \({^g G_i}\), and hence \({^g G_i} = H\).
\end{proof}
 
\section{Profinitely rigid words}\label{sect:rigid-words}

If \(F\) is a free group of rank \(d\) and \(w \in F\) is a word, then for any finite group \(G\) one has a word map \[w\colon \underbrace{G \times \cdots \times G}_{d} \to G\] given by substitution. The pushfoward of the uniform measure on \(G^d\) to \(G\) defines the \(w\)-measure on \(G\). Two words \(w_1\) and \(w_2\) in \(F\) are profinitely equivalent if the \(w_1\)- and \(w_2\)-measures coincide in every finite group \(G\). Observe that this is automatically the case if \(w_1\) and \(w_2\) lie in the same \(\Aut(F)\)-orbit. The following result gives an alternative description of profinitely equivalent words, and also justifies the definition of profinite rigidity of a word:

\begin{pro}[{\cite[Thm. 2.2]{hanany_orbits_2020}}] Let \(w_1\) and \(w_2\) be two elements of \(F\). The following are equivalent:
\begin{enumerate}
    \item The \(w_1\)-measure and the \(w_2\)-measure coincide in every finite group.
    \item There is an automorphism \(\varphi\in \Aut(\widehat{F})\) such that \(\varphi(w_1) = w_2\).
\end{enumerate}
\end{pro}

We recall that \(w \in F\) is profinitely rigid if any \(u \in F\) in the same \(\Aut(\widehat{F})\)-orbit of \(w\) also lies in the same \(\Aut(F)\)-orbit of \(w\). It is universally profinitely rigid if it is profinitely rigid in \(F*H\) for any finitely generated free group \(H\).

Following \cite[Sec. 3]{Miasnikov2007}, we say that an element \(w\) of a free group \(F\) is \emph{algebraic} if there is no proper free factor of \(F\) that contains \(w\). Given an element \(w \in F\), there exists a unique free factor \(U\) of \(F\) that contains \(w\) such that \(w\) is algebraic in \(U\) (\cite[Sec. 3.4]{Miasnikov2007}). As in \cite{Miasnikov2007}, we call \(U\) the \emph{algebraic closure} of \(w\) in \(F\). The following proposition shows that the algebraic closure of a word in a free group is a profinite invariant.

\begin{pro}\label{algebraic} Let \(F\) be a finitely generated free group and let \(w \in F\).  Let \(L\) be a free profinite factor of \(\widehat{F}\) that contains \(w\). Then the algebraic closure of \(w\) in \(F\) is contained in \(L\). In particular, if \(U\) is the algebraic closure of \(w\), then \(\overline{U}\) is the smallest profinite free factor of \(\widehat{F}\) that contains \(w\).
\end{pro}

\begin{proof}
Write \(\widehat{F} = L \amalg K\). Then, by \cref{relative splitting}, \(F\) splits as a free product \(F = G_1 * G_2\) such that \(w \in G_1\) and \(G_1\) is contained in some \(\widehat{F}\)-conjugate of \(K\) or \(L\). Since \(w \in L\), the only possibility is \(G_1 \leq L\). Hence, the algebraic closure of \(w\) in \(F\) is contained in \(L\).
\end{proof}

We are now ready to prove \cref{theorem-profinitely-rigid-words}.

\begin{proof}[Proof of \cref{theorem-profinitely-rigid-words}] It suffices to prove the theorem for \(w\) algebraic in \(F\). Assume that \(w\) is profinitely rigid in \(F\), and let \(u \in F*H\) be such  that there exists \(\varphi \in \Aut(\widehat{F*H})\) with \(\varphi(w)= u\). We want to show that \(u\) and \(w\) are in the same \(\Aut(F*H)\)-orbit. 

Let \(U\) be the algebraic closure of \(u\) in \(F*H\). Then \(F*H = U* H'\), and by \cref{algebraic} we have \(\varphi(\overline{F}) = \overline{U}\), where \(\overline{F}\) and \(\overline{U}\) denote the closure of \(F\) and \(U\) respectively in \(\widehat{F*H}\). Hence, \(\operatorname{rk}(F) = \operatorname{rk}(U)\), and so \(\operatorname{rk}(H) = \operatorname{rk}(H')\). Therefore, there exists an automorphism of \(F *H\) sending \(U\) to \(F\). Hence, we may assume that \(U = F\). Since \(w\) is profinitely rigid in \(F\) and, as already noted, \(\varphi(\overline{F}) = \overline{F}\), it follows that \(u\) and \(w\) are in the same \(\Aut(F)\)-orbit.

Conversely, assume that \(w\) is profinitely rigid in \(F *H\), and let \(u \in F\) be such that there exists \(\varphi \in \Aut(\widehat{F})\) with \(\phi(w) = u\). Any automorphism of \(\widehat{F}\) extends to an automorphism of \(\widehat{F*H}\), and thus we can find \(\psi \in \Aut(F *H)\) such that \(\psi(w) = u\). Let \(U\) be the algebraic closure of \(u\) in \(F *H\). Then, by \cref{algebraic}, \(\psi(F) = U \leq F\). Since \(\psi\) is an automorphism, we have \(\psi(F) = F\), showing that \(w\) is profinitely rigid in \(F\) as desired.
\end{proof}

\section{Further directions}\label{sect:limit}

Past the fundamental groups of graphs of virtually free groups with virtually cyclic edge groups, it is natural to try to extend \cref{thm-detects-free-products} to the class \(\mathcal{L}\) of limit groups, described as follows.

Let \(G'\) be a group and let \(A \subseteq G'\) be a finite malnormal family. Let \(Y\) be a graph with vertex set \(V(Y) = \{u\} \cup A\) and edge set \(E(Y) = \{e_a\colon a \in A\}\) where each edge \(e_a\) satisfies \(\iota(e_a) = u\) and \(t(e_a) = a\). Define the graph of groups \(\G\) on \(Y\) by means of \(\G(u) = G'\), \(\G(a) = \langle a \rangle \times \bbz^{n_a}\) for some \(n_a > 0\) and \(\G(e_a) = \langle a \rangle\). The homomorphisms from the edge groups of \((Y,\G)\) into the vertex groups are taken to be natural inclusions. The group \(\pi(Y,\G)\) is said to be obtained from \(G'\) by an extension of centralizers.

Let \(\mathcal{D}_0\) consists of the class of free groups, and define \(\mathcal{D}_{i+1}\) to be the class of groups that can be obtained from groups in \(\mathcal{D}_i\) by an extension of centralizers. We denote by \(\mathcal{L}_i\) the class of finitely generated subgroups of groups in \(\mathcal{D}_i\). Then, a \emph{limit group} is any group in the class \(\mathcal{L} = \bigcup_{i \geq 0} \mathcal{L}_i\).

\begin{Conj} The profinite completion of a one-ended limit group does not split as a free profinite product.\end{Conj}

Our methods allows one to obtain particular cases of this conjecture. By a result of Wilton, the limit groups are LERF (\cite{Wi08}). The same argument as in our proof of \cref{main-thm} provides:

\begin{teo}\label{limit1}
    Profinite completions detect free products in \( \mathcal{L}_1 \) .
\end{teo}
\begin{proof} Observe that limit groups are torsion-free, and any Baumslag--Solitar subgroup of a limit group must be isomorphic to \(\bbz^2\), so that the last conclusion of \cref{BSsubgroups} still applies. Moreover, \cref{claim-after-collapsing} remains correct as any incompressible edge incident to an abelian vertex must be parabolic by \cite[Prop. 4.2.9]{ribesProfiniteGraphsGroups2017} -- in fact, the abelian vertex group must itself be contained in the same \(\widehat{G}\)-conjugate of \(K\) or \(L\). The rest of the proof follows without changes. 
\end{proof}

%

\bibliographystyle{alpha}
\bibliography{biblio}

\end{document}